\documentclass[11pt]{amsart}
\usepackage{amscd,amssymb}
\usepackage{amsfonts}
\usepackage{amsmath}
\usepackage{mathrsfs}
\usepackage{graphicx} 
\usepackage{color}
\usepackage{hyperref} 

\newtheorem{theorem}{Theorem}[section]
\newtheorem*{theorem*}{Theorem}

\newtheorem{lemma}[theorem]{Lemma} 
 
\newtheorem{definition}[theorem]{Definition}

\newtheorem{remark}[theorem]{Remark}
\title[Counting filter restricted paths ]{Counting filter restricted paths in $\mathbb{Z}^2$  lattice }

\author{Olga Postnova}
\address{O.P.: Euler International Mathematical Institute, Laboratory of Mathematical Problems of Physics, St. Petersburg Department of Steklov Institute of Mathematics, Saint Petersburg, Fontanka river emb. 27,
191023 Saint Petersburg,
Russia}
\email{postnova.olga@gmail.com}
\author{Dmitry Solovyev}
\address{D.S.: Yau Mathematical Sciences Center, Tsinghua University, Beijing 100084, China}
\email{dimsol42@gmail.com}
\begin{document}

\maketitle

\begin{abstract}
We derive a path counting formula for two-dimensional lattice path model on a plane with filter restrictions. A filter is a line  that restricts the path passing it to one of possible directions. Moreover, each path that touches this line is assigned a special weight. The periodic filter restrictions are motivated by the problem of tensor power decomposition for representations of  quantum $\mathfrak{sl}_2$ at roots of unity. Our main result is the explicit formula for the weighted number of paths from the origin to a fixed point between two filters in this model.
\end{abstract}

\tableofcontents

\section*{Introduction}\label{intro}
Counting lattice paths is one of the central problems in combinatorics \cite{Kratt}. It provides a powerful tool for the problems arising in representation theory of Lie algebras such as counting lattice paths in Weyl chambers (\cite{GM}, \cite{G}, \cite{TZ}). In this paper we count paths on Bratteli diagram\cite{B} with restrictions of two types which we call filters. This problem is motivated by the structures arising in representation theory of quantum groups at roots of unity (\cite{RT}, \cite{And}). Lattice path model explored in the present paper serves as a prototype for models, where weighted numbers of paths reproduce recurrence relations for multiplicities in tensor product decomposition of representations for different versions of quantized universal enveloping algebras of Lie algebra $\mathfrak{sl}_2$ at roots of unity, mainly, $U_q({\mathfrak{sl}_2})$ with divided powers and small quantum group $u_q({\mathfrak{sl}_2})$. The full treatment of representation theoretic part and asymptotic analysis will be carried out in \cite{LPRS}.

\begin{figure}[h!]
	\centerline{\includegraphics[width=220pt]{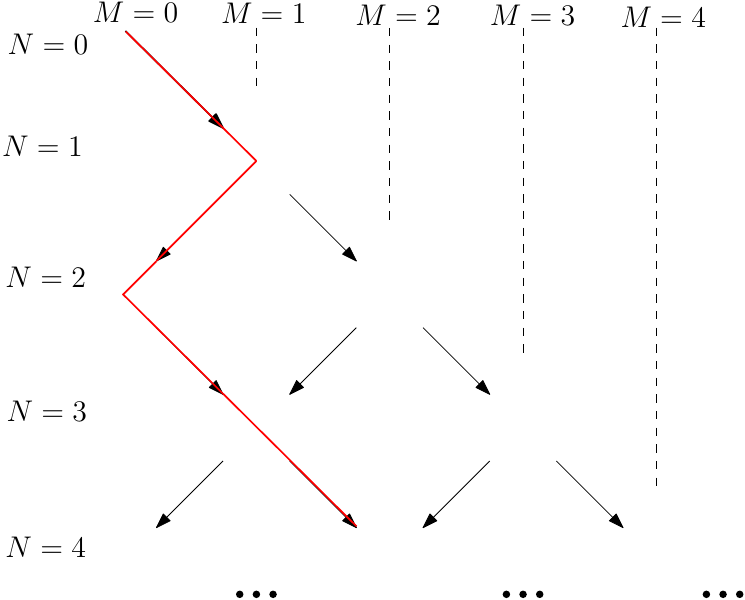}}
	\caption{Bratteli diagram and a lattice path.}
	\label{B}
\end{figure}

An example of a path on Bratteli diagram is shown on Figure \ref{B}. It is well known that if there is no constraint for paths to go to negative $M$ then the number of lattice paths arriving to point $(M,N)$ from $(0,0)$ is
\begin{equation*}
   \binom{N}{\frac{N-M}{2}}
\end{equation*}
for $N\geq M\geq 0$. Since in Bratteli diagram paths can not go into the region with $M<0$ the number of paths arriving to point $(M,N)$ from the origin is given by the reflection principle (see Section \ref{wall}):
\begin{equation*}
   \binom{N}{\frac{N-M}{2}}-\binom{N}{\frac{N-M}{2}-1}.
\end{equation*}
A similar method was introduced by André \cite{Andre}  in 19th century to solve the two candidate ballot problem  by counting unfavorable records and subtracting them from the total number of records. The term 'reflection principle' was attributed to André  in the books of M. Feller \cite{F} and J. L. Doob \cite{D}. For the detailed history of the reflection principle, see \cite{R}. The reflection principle has been one of the key elements in various lattice path models providing explicit enumerative formulas. It  was generalized by Gessel and Zeilberger \cite{GZ} to lattice walks on  Weyl chambers, which are the regions preserved under the actions of Weyl reflection groups.  The above mentioned case corresponds to counting paths in Weyl chamber of type $A_1$.\par
In this paper we will introduce a class of restrictions  on paths on lattice $\mathcal{L}=\{(n,m)| n+m=0\; \text{mod} \;2\}\subset \mathbb{Z}^2$ motivated by tensor product rule of indecomposable modules of $U_q({\mathfrak{sl}_2})$ with divided powers and small quantum group $u_q({\mathfrak{sl}_2})$. Our main result is the explicit formula for the number of lattice paths with periodic filter restrictions.\par
The paper is organized as follows. In Sections \ref{notations}, \ref{maintheorem} and \ref{unrestrict} we give the description of the lattice path model and formulate the main theorem. In Section \ref{onerestrict} we define wall and filter restrictions and recall the reflection principle. In Section \ref{wallfilter} we will reduce the problem of counting paths between the wall and the filter to a problem of counting paths between two lines. In Sections \ref{twofilters}, \ref{walltwofilters} we will prove theorems for path counting in the presence of two filters and two filters together with the wall. The proof of the main theorem is given in Section \ref{multfilt}. In Section \ref{conclusions} we hint at possible applications of considered lattice path models to representation theory of quantum groups at roots of unity.

\subsection*{Acknowledgements} We are grateful to A. Lachowska, N. Reshetikhin, F. Petrov and C. Krattenthaler for many valuable discussions. The work was supported by the Ministry of Science and Higher Education of the Russian Federation (agreement no. 075–15–2022–287). The work of D.S. was supported by the Xing Hua Scholarship program of Tsinghua University.\par

\section{Notations} \label{notations}
In this paper we will use notations following \cite{Kratt}. Throughout this paper we will be considering the lattice
$$\mathcal{L}=\{(n,m)| n+m=0\; \text{mod} \;2\}\subset \mathbb{Z}^2,$$
and the set of steps $\mathbb{S}=\mathbb{S}_L\cup \mathbb{S}_R$, where
$$\mathbb{S}_R=\{(x,y)\to(x+1,y+1)\},\quad \mathbb{S}_L=\{(x,y)\to(x-1,y+1)\}.$$
A \textit{lattice path} $\mathcal{P}$ in $\mathcal{L}$ is a sequence $\mathcal{P}=(P_0,P_1,
\dots, P_m)$ of points $P_i=(x_i,y_i)$ in $\mathcal{L}$ with starting point $P_0$ and the endpoint $P_m$. 
The pairs $P_0\to P_1, P_1\to P_2\dots P_{m-1}\to P_m$ are called steps of $\mathcal{P}$.

Given starting point $A$ and endpoint $B$, a set $\mathbb{S}$ of steps and
a set of restrictions $\mathcal{C}$ we write
$$L(A\to B\;|\;\mathcal{C})$$
for the set of all lattice paths from $A$ to $B$ that have steps from $\mathbb{S}$ and obey the restrictions from $\mathcal{C}$. Since we only consider the set of steps $\mathbb{S}$ we will omit it from the notations. We will denote the number of paths in this set as
$$|L(A\to B\;|\;\mathcal{C})|.$$ 

The set of restrictions $\mathcal{C}$ that are considered will be called "filter restrictions" because they forbid steps in certain directions and provide other steps with non uniform weights. To each step from $(x,y)$ to $(\tilde{x},\tilde{y})$  we will assign the weight  function $\omega:\mathbb{S}\longrightarrow \mathbb{R}_{>0}$ and use notation  $(x,y)\xrightarrow[]{\omega}(\tilde{x},\tilde{y})$  to denote that  the step  from $(x,y)$ to $(\tilde{x},\tilde{y})$ has the weight $\omega$. By default, all unrestricted steps from $\mathbb{S}$ will have weight $1$  and will denoted by an arrow with no number at the top.
 The \textit{weight} of a path $\mathcal{P}$ is defined as the product
$$\omega(\mathcal{P})=\prod_{i=0}^{m-1}\omega(P_i\to P_{i+1}).$$

For the set $L(A\to B\;|\;\mathcal{C})$ we define the \textit{weighted number of paths} as
$$Z(L(A\to B\;|\;\mathcal{C}))=\sum_{\mathcal{P}}\omega(\mathcal{P}),$$
where the sum is taken over all paths $\mathcal{P}\in L(A\to B\;|\;\mathcal{C})$. This is a partition function of a random walk originated in $A$ and ending at $B$ with
$$Prob(\mathcal{P})=\frac{\omega(\mathcal{P})}{Z(A\to B\;|\;\mathcal{C})}.$$
When $\omega(s)=1$ for all of the allowed steps,
$$Z(L(A\to B;\mathbb{S}\;|\;\mathcal{C}))=|L(A\to B;\mathbb{S}\;|\;\mathcal{C})|.$$

\section{Main theorem}	
\label{maintheorem}
We will be interested in counting weighted number of paths in the set of paths on $\mathcal{L}$ with steps from $\mathbb{S}$ that start at $(0,0)$ and end at $(M,N)$ in the presence of the arrangement of restrictions which we will call the \textit{left wall restriction} $\mathcal{W}_0^L$, the \textit{filter restriction} $\mathcal{F}_{l-1}^1$ \textit{of type 1} and the \textit{filter restrictions} $\mathcal{F}_{nl-1}^2$ \textit{of type 2}, where $l$ is a fixed parameter $l\in\mathbb{Z}\geq 2$, and $n=2,3,\dots$ (see Figure \ref{fill}). The upper index denotes the type of restriction and the lower index denotes its position on $\mathcal{L}$.
We will denote this set of paths by 
\begin{equation*}
L_N((0,0)\to (M,N)\;|\;\mathcal{W}_0^L, \mathcal{F}_{l-1}^1,\{\mathcal{F}_{(n+1)l-1}^{2}\}, n\in \mathbb{Z}_+).
\end{equation*}

The \textit{left wall restriction} $\mathcal{W}^L_0$ is located at $x=0$ and implies that at points $(0,y)$ only one step is allowed:
$$
\mathcal{W}^L_{0}=\{(0,y)\to (1,y+1)\}.
$$

The \textit{filter restriction} $\;\mathcal{F}_{l-1}^{1}$ \textit{of type 1} is located at $x=l-1$ and implies that at $x=l-1,\;l$ only the following steps are allowed:
$$
\mathcal{F}_{l-1}^{1}=\{(l-1,y)\to(l,y+1),\; (l,y+1)\to(l+1,y+2),\; (l,y+1)\xrightarrow[]{2}(l-1,y+2)\}.
$$

The \textit{filter restriction} $\mathcal{F}_{nl-1}^{2}$  \textit{of type 2} is located at $x=nl-1$ and implies that at $x=nl-2,\;nl-1,\;nl$ only the following steps are allowed:
\begin{eqnarray*}
\mathcal{F}_{nl-1}^{2}&=&\{(nl-2,y-1)\xrightarrow[]{2}(nl-1,y),\; (nl-2,y-1)\to(nl-3,y),\\
& &(nl-1,y)\to(nl,y+1),\; (nl,y+1)\to(nl+1,y+2),\; (nl,y+1)\xrightarrow[]{2}(nl-1,y+2)\}.
\end{eqnarray*}

We will denote by \textit{multiplicity function in the $j$-th strip} the weighted number of paths in this set  with the endpoint $(M,N)$ that lies within  $(j-1)l-1 \leq M< jl-2$
\begin{equation*}
	M^j_{(M,N)} = Z(L_N((0,0)\to (M,N)\;|\;\mathcal{W}_0^L, \mathcal{F}_{l-1}^1,\{\mathcal{F}_{(n+1)l-1}^{2}\}, n\in \mathbb{Z}_+)),
\end{equation*}
where $M\geq 0$, $j\geq 2$ and $j=\Big[\frac{M+1}{l}+1\Big]$.

The main result is the Theorem \ref{tj}:
\begin{theorem*}
	The multiplicity function in the $j$-th strip  is given by
	\begin{eqnarray*}
	M^{j}_{(M,N)} = 2^{j-2}\Big(\sum_{k=0}^{\big[\frac{N-(j-1)l+1}{4l}\big]}P_j(k) F^{(N)}_{M+4kl}+\sum_{k=0}^{\big[\frac{N-jl}{4l}\big]}P_j(k)F^{(N)}_{M-4kl-2jl} -\\-\sum_{k=0}^{\big[\frac{N-(j+1)l+1}{4l}\big]}Q_j(k)F^{(N)}_{M+2l+4kl}-\sum_{k=0}^{\big[\frac{N-(j+2)l}{4l}\big]}Q_j(k)F^{(N)}_{M-4kl-2(j+1)l} \Big),
	\end{eqnarray*}
	where
	\begin{equation*}
	P_j(k)=\sum_{i=0}^{\big[\frac{j}{2}\big]}\binom{j-2}{2i}\binom{k-i+j-2}{j-2},\;\;\;\;\;
	Q_j(k)=\sum_{i=0}^{\big[\frac{j}{2}\big]}\binom{j-2}{2i+1}\binom{k-i+j-2}{j-2},
	\end{equation*}
\begin{equation*}
F_M^{(N)}= {N\choose\frac{N-M}{2}}-{N\choose{\frac{N-M}{2}-1}}.
\end{equation*}
\end{theorem*}

\begin{figure}[h!]
	\centerline{\includegraphics[width=350pt]
		{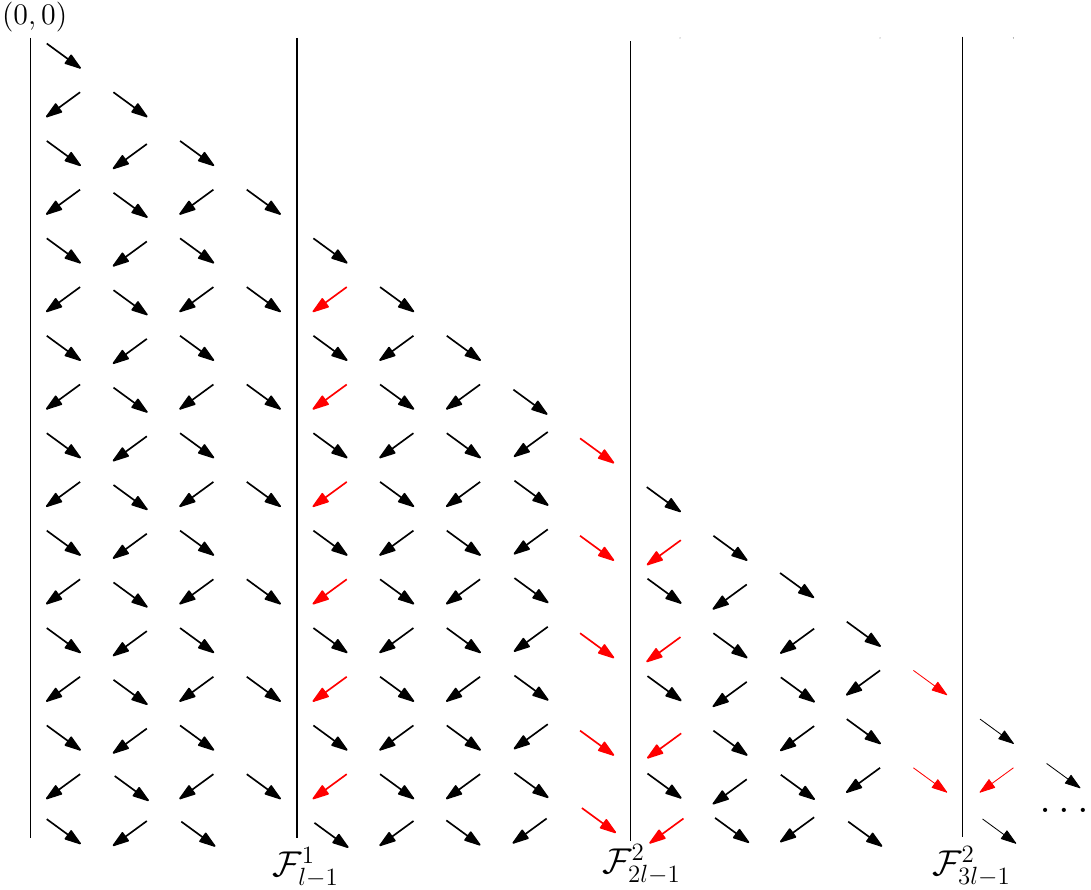}
	}
	\caption{The arrangement of restrictions on $\mathcal{L}$: the wall $\mathcal{W}_0^L$, the filters $\mathcal{F}_{l-1}^1$ and the filters $\mathcal{F}_{nl-1}^2$, where $n=2,3,\dots$ and $l=5$. Red arrows correspond to steps with weight 2.
	}
	\label{fill}
\end{figure}

\section{Unrestricted lattice paths} 
\label{unrestrict}
In this section we will recall general formulas for unrestricted paths counting. For further details see \cite{Kratt}.  
Below we will use the notation  $$L(A\to B)$$  for the  set of unrestricted  paths  from $A$ to $B$  on $\mathcal{L}$ with the set of steps $\mathbb{S}$. An example of unrestricted path  on $\mathcal{L}$ is shown on Figure \ref{sk}. 

\begin{figure}[h!]
	\centerline{\includegraphics[height=160pt]{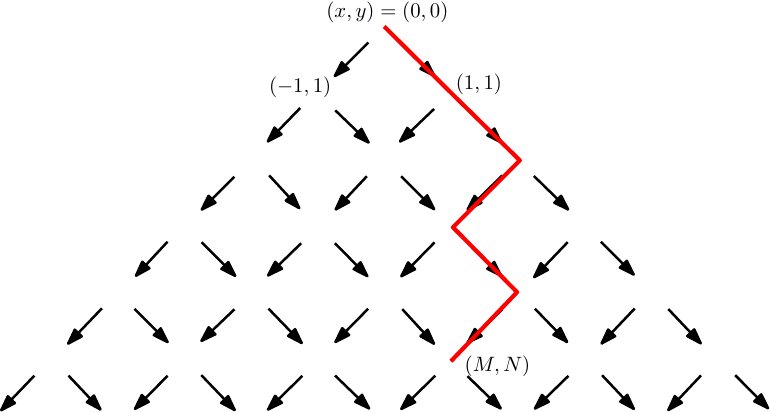}}
	\caption{Path in $L((0,0)\to (M,N))$.}
	\label{sk}
\end{figure}

\begin{lemma}
\label{l1}
The number of paths in the set $L((0,0)\to (M,N))$ is given by 
$$|L((0,0)\to (M,N))|={N\choose\frac{N-M}{2}}$$
\begin{proof}
$|L((0,0)\to (M,N))|$ satisfies the same recurrence relations and initial conditions as ${N\choose\frac{N-M}{2}}$.
\end{proof}
\end{lemma}
Due to the translation invariance we can also count lattice paths originated at the point $(i,j)$  because we have a natural bijection:
$$L((0,0)\to (M,N))\simeq L((i,j)\to (M+i,N+j)).$$
Therefore
$$|L((0,0)\to (M,N))|=|L((i,j)\to (M+i,N+j))|.$$

\section{Counting paths with one restriction}
\label{onerestrict}
In this section we will count the number of paths from $A$ to $B$ on $\mathcal{L}$ with the set of steps $\mathbb{S}$ with one restriction $\mathcal{C}$. Denote such set of paths as
$$L(A\to B|\;\mathcal{C}).$$
\subsection{Wall restriction}
\label{wall}
\begin{definition} For lattice paths that start at $(0,0)$ we will say that $\mathcal{W}^L_d$ with $d\leq 0$ is a \textit{left wall restriction} (relative to $x=0$) if at points $(d,y)$ paths are allowed to take steps of type $\mathbb{S}_R$ only:
$$\mathcal{W}^L_{d}=\{(d,y)\to (d+1,y+1)\}.$$
 
Similarly, we will say that $\mathcal{W}^R_{d}$ with $d\geq 0$  is a \textit{right wall restriction} (relative to $x=0$) if in points $(d,y)$ paths are allowed to take  steps from $\mathbb{S}_L$ only:
$$\mathcal{W}^R_{d}=\{(d,y)\to (d-1,y+1)\}.$$ 
\end{definition}  
\begin{lemma}
\label{l2}
The number of paths from $(0,0)$ to $(M,N)$ with the set of steps $\mathbb{S}$ and one wall restriction $\mathcal{W}^L_{a}$ or $\mathcal{W}^R_{b}$ can be expressed via the number of unrestricted paths as

\begin{equation}\label{lw}
|L((0,0)\to (M,N)\;|\;\mathcal{W}^L_a)|={N\choose\frac{N-M}{2}}-{N\choose\frac{N-M}{2}+a-1},\;\;\text{ for}\;\; M\geq a,
\end{equation}
\begin{equation}\label{rw}
|L((0,0)\to (M,N)\;|\;\mathcal{W}^R_{b})|={N\choose\frac{N-M}{2}}-{N\choose\frac{N-M}{2}+b+1}, \;\;\text{ for}\;\; M\leq b.
\end{equation}

\begin{proof}
Let us give a brief proof of the first statement via the reflection principle \cite{Kratt}. The proof of the second statement is completely similar.

In order to enumerate $|L((0,0)\to (M,N)\;|\;\mathcal{W}^L_{a})|$ we embed it into the bigger set of unrestricted paths with steps $\mathbb{S}_R\cup\mathbb{S}_R$ which originate at $(0,0)$ and $(2(a-1),0)$ as is shown on Figure \ref{lwf}.

\begin{figure}[h!]
	\centerline{\includegraphics[height=200pt]{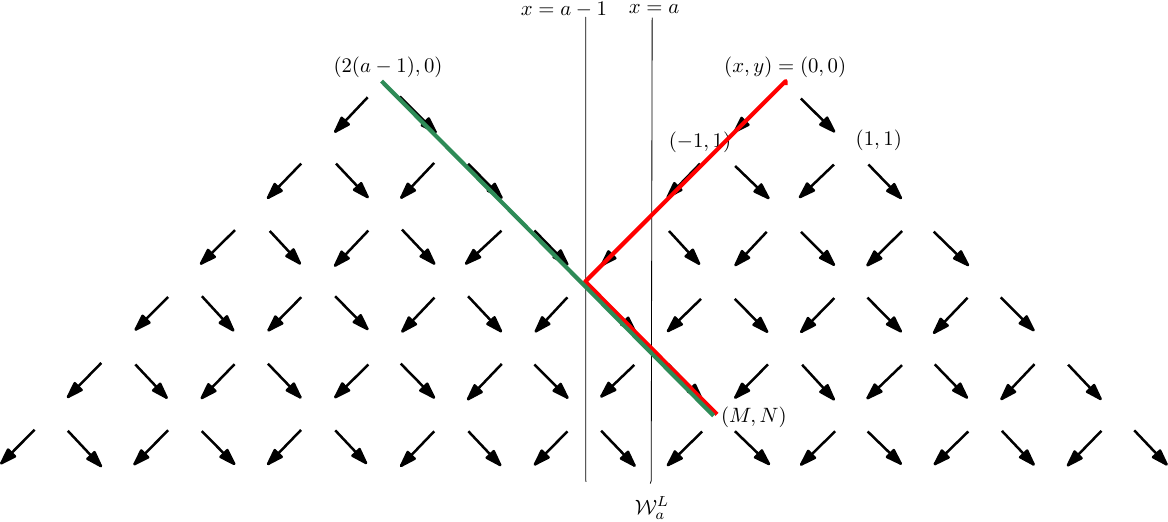}}
	\caption{Counting paths with the wall restriction. The initial path $\mathcal{P}$(red) and partially reflected path $\phi \mathcal{P}$(green).}
	\label{lwf}
\end{figure}

Unrestricted path $\mathcal{P}$ originating at $(0,0)$ will either be reflected from $\mathcal{W}^L_{a}$ or will cross the line $x=a$ to the left. If $\mathcal{P}$ crosses the wall, define by $\phi\mathcal{P}$ the path which coincide with $\mathcal{P}$ after the first wall crossing and its part before the first crossing is reflected with respect to the line $x=a-1$. The path $\phi\mathcal{P}$ originates at $(2(a-1),0)$.

The set $L((0,0)\to (M,N))$ of unrestricted paths is the disjoint union
\begin{equation}
\label{refl}
L((0,0)\to (M,N))=S_{+}\sqcup S_{-},
\end{equation}
where $S_{+}$ are the paths with wall restriction $\mathcal{W}^L_{a}$ and $S_{-}$ are the paths crossing the wall. By the observation above we can identify
$$\phi S_{-}=L((2(a-1),0)\to (M,N)),$$
and therefore 
$$|S_{-}|=|\phi S_{-}|=|L((2(a-1),0)\to (M,N))|={N\choose\frac{N-M}{2}+a-1}.$$
The formula ($\ref{refl}$) implies that
$$|S_{+}|=|L((0,0)\to (M,N))|-|S_{-}|.$$
and thus we have proved (\ref{lw}). The proof of (\ref{rw}) is completely parallel.
\end{proof}
\end{lemma}

\subsection{Filter restriction of type 1}
\begin{definition}
	We will say that there is a filter $\mathcal{F}_d^{1}$ of type 1, located at $x=d$ if at $x=d,\;d+1$ only the following steps are allowed:
	$$\mathcal{F}_d^{1}=\{(d,y)\to(d+1,y+1),\; (d+1,y+1)\to(d+2,y+2) ,\; (d+1,y+1)\xrightarrow[]{2}(d,y+2)\}.$$
The index above the arrow is the weight of the step. By default, an arrow with no number at the top means that the corresponding step has weight $1$.
\end{definition}
\begin{figure}[h!]
	\centerline{\includegraphics[width=150pt]{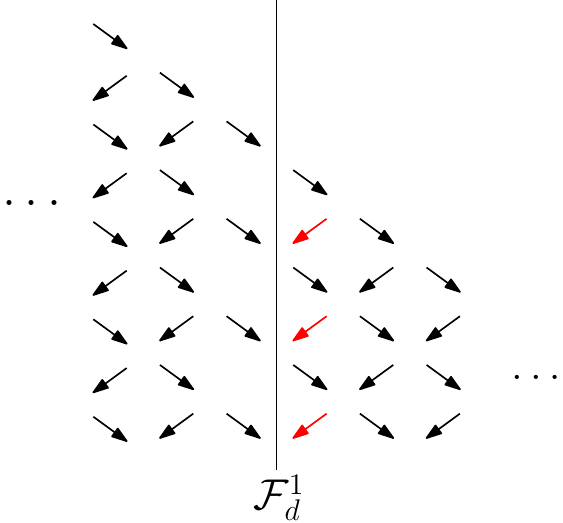}}
	\caption{Filter $\mathcal{F}_d^{1}$. Red arrows correspond to step $(d+1,y+1)\xrightarrow[]{2}(d,y+2)$ that has a weight $2$. Other steps have weight $1$.}
\end{figure}
\begin{lemma}
\label{l3}
The number of lattice paths from $(0,0)$ to $(M,N)$ with steps from $\mathbb{S}$ and filter restriction $\mathcal{F}_d^{1}$ with $x=d>0$ is
$$
|L_N((0,0)\to (M,N)\;|\;\mathcal{F}_d^{1})|={N\choose\frac{N-M}{2}}-{N\choose\frac{N-M}{2}+d},\;\;\text{ for}\;\; M<d .
$$ 
In other words, the number of paths that start at  point $(0,0)$ to the left of $\mathcal{F}_d^{1}$ and end at  point $(M,N)$ to the left of $\mathcal{F}_d^{1}$ is equal to the number of paths  from $(0,0)$ to $(M,N)$ with right wall restriction $\mathcal{W}_{d-1}^{R}$.
\begin{proof}
We proceed by the reflection principle similarly to the proof of Lemma \ref{l2}. The step that is forbidden by the filter is $(d,y)\xrightarrow[]{}(d-1,y+1)$. Hence, any path that crosses the axis $x=d$ can not return back to the region $x<d$. To count these paths we establish a bijection between them and paths starting at $(2d,0)$. We do so by reflecting the portion of each path until its first arrival at $x=d$ and leaving the other part unchanged. These paths need to be excluded from the total number of unrestricted paths, therefore 
\begin{eqnarray*}
|L_N((0,0)\to (M,N)\;|\mathcal{F}_d^{1})| &=& |L_N((0,0)\to (M,N)\;|-|L_N((2d,0)\to (M,N)\;|=\\
&=& {N\choose\frac{N-M}{2}}- {N\choose\frac{N-M}{2}+d}.
\end{eqnarray*}
\end{proof}
\end{lemma}
\begin{lemma}
\label{l4}
The weighted number of lattice paths from $(0,0)$  to $(M,N)$  with steps from $\mathbb{S}$ and filter restriction $\mathcal{F}_d^{1}$ with $d>0$
is
$$
Z(L_N((0,0)\to (M,N)\;|\;\mathcal{F}_d^{1}))={N\choose\frac{N-M}{2}},\;\;\text{ for}\;\; M\geq d.
$$ 
In other words, the weighted number of paths that start at point $(0,0)$ to the left of $\mathcal{F}_d^{1}$ and end at point $(M,N)$ to the right of $\mathcal{F}_d^{1}$ is equal to the number of unrestricted paths from $(0,0)$ to $(M,N)$.
\begin{proof}
For brevity denote $X=L_N((0,0)\to (M,N)\;|\;\mathcal{F}_d^{1})$, where $M\geq d$. Define 
$$
\psi: X\to L_N((0,0)\to (M,N))
$$ 
that acts on a path $\mathcal{P}_{m}\in X$ which has $m$ steps $(d+1,y+1)\xrightarrow[]{2}(d,y+2)$ and a weight $2^{m}$ and produces $2^{m}$ paths in the set of unrestricted paths $L_N((0,0)\to (M,N);\mathbb{S})$ that have weight $1$:
$$
\psi(\mathcal{P}_m)=\{\mathcal{P}_1^1,\dots,\mathcal{P}_1^{2^m}\}.
$$
First, it reflects a portion of $\mathcal{P}_{m}$ between its two last visits to $x=d$ and produces two paths with weight $2^{m-1}$. Then we do the same to these two paths, where the next portions of such paths are reflected in a similar manner. We repeat this procedure until all of the produced paths are of weight $1$. Example of this procedure for $m=1$ is depicted in Figure \ref{l1fig}.
\begin{figure}[h!]
	\centerline{\includegraphics[height=160pt]{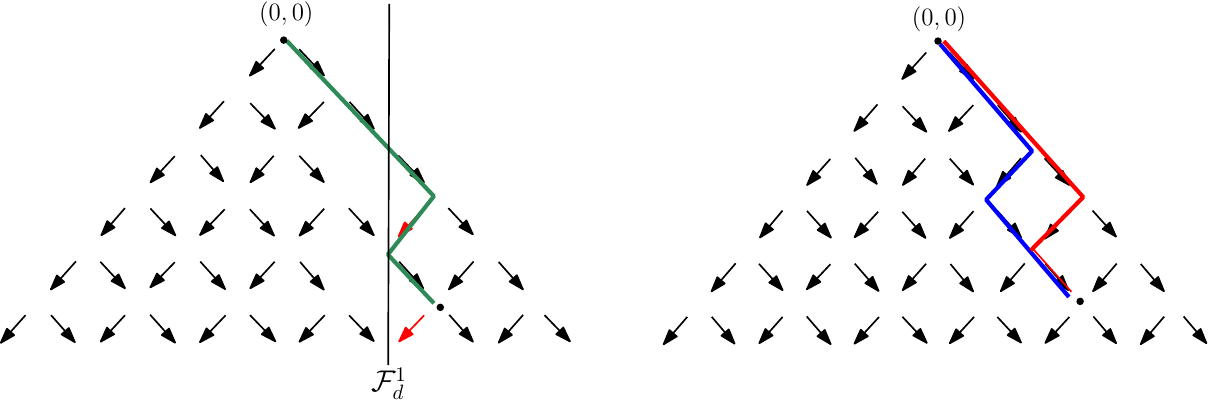}}
	\caption{Counting weighted paths that start on the left of $\mathcal{F}_d^{1}$ and end on the right of $\mathcal{F}_d^{1}$ via counting unrestricted paths.} 
	\label{l1fig}
\end{figure}\\
The paths $\tilde{\mathcal{P}}$ that do not contain the step $(d+1,y+1)\xrightarrow[]{2}(d,y+2)$  will not be reflected:
$$
\psi(\tilde{\mathcal{P}})=\tilde{\mathcal{P}}.
$$ 
From the result of this procedure we can see that the problem of finding the weighted number of paths $Z(X)$ is equal to counting the number of non-weighted paths in an unrestricted case (without the filter $\mathcal{F}_d^{1}$), therefore $Z(X)=|L_N((0,0)\to (M,N))|={N\choose\frac{N-M}{2}}$.
\end{proof}
\end{lemma}

\begin{lemma}
\label{l5}
The weighted number of lattice paths from $(0,0)$ to $(M,N)$ with steps from $\mathbb{S}$ and filter restriction $\mathcal{F}_{-d}^{1}$ with $d\geq0$ is
$$
Z(L_N((0,0)\to (M,N)\;|\;\mathcal{F}_{-d}^{1}))={N\choose\frac{N-M}{2}}+ {N\choose\frac{N-M}{2}+d},\;\;\text{ for}\;\; M\geq -d.
$$ 
In other words, the weighted number of  paths that start at point $(0,0)$ to the right of $\mathcal{F}_{-d}^{1}$ and end at  point $(M,N)$ to the right of $\mathcal{F}_{-d}^{1}$ is equal to the number of unrestricted paths from $(0,0)$ to $(M,N)$ and from $(-2d,0)$ to $(M,N)$.
\begin{proof}
Similar to Lemma \ref{l4}.
\end{proof}
\end{lemma}
\subsection{Filter restriction of type 2}
\begin{definition}
We will say that there is a filter $\mathcal{F}_d^{2}$ of type 2, located at $x=d$ if at $x=d-1,\;d,\;d+1$ only the following steps are allowed: 
\begin{eqnarray*}
	\mathcal{F}_{d}^{2}&=&\{(d-1,y-1)\xrightarrow[]{2}(d,y),\; (d-1,y-1)\to(d-2,y),\\
	& &(d,y)\to(d+1,y+1),\; (d+1,y+1)\to(d+2,y+2),\; (d+1,y+1)\xrightarrow[]{2}(d,y+2)\}.
\end{eqnarray*}
The index above the arrow is the weight of the step. By default, an arrow with no number at the top means that the corresponding step has weight $1$.
\end{definition}
\begin{figure}[h!]
	\centerline{\includegraphics[width=150pt]{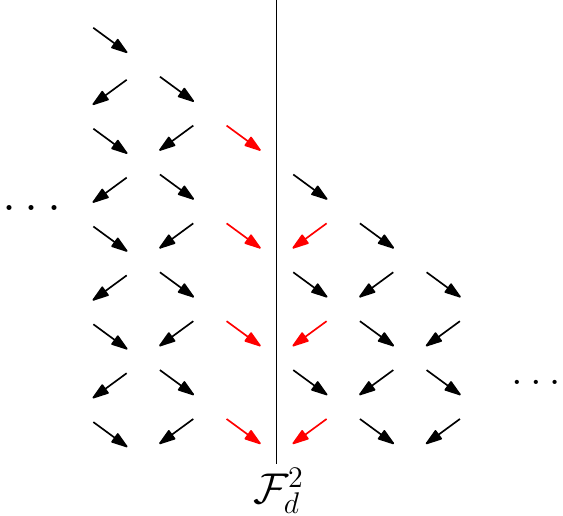}}
	\caption{Filter $\mathcal{F}_d^{2}$. Red arrows correspond to steps that have a weight $2$. Other steps have weight $1$.}
\end{figure}

\begin{lemma}
\label{l6}
The number of lattice paths from $(0,0)$  to $(M,N)$  with steps from $\mathbb{S}$ and filter restriction $\mathcal{F}_d^{2}$ 
is
$$
|L_N((0,0)\to (M,N)\;|\;\mathcal{F}_d^{2})|={N\choose\frac{N-M}{2}}- {N\choose\frac{N-M}{2}+d},\;\;\text{ for}\;\; M<d .
$$ 
In other words the number of paths that start at  point $(0,0)$ to the left of $\mathcal{F}_d^{2}$ and end at  point $(M,N)$ to the left of $\mathcal{F}_d^{2}$ is equal to the number of paths  from $(0,0)$ to $(M,N)$ with right wall restriction $\mathcal{W}_{d-1}^{R}$.
\begin{proof}
Similar to Lemma \ref{l3}.
\end{proof}
\end{lemma}
\begin{lemma}
\label{l7}
The weighted number of lattice paths from $(0,0)$  to $(M,N)$  with steps from $\mathbb{S}$ and filter restriction $\mathcal{F}_d^{2}$ 
is
$$
Z(L_N((0,0)\to (M,N)\;|\;\mathcal{F}_d^{2}))=2 {N\choose\frac{N-M}{2}},\;\;\text{ for}\;\; M>d.
$$ 
In other words, the weighted number of paths that start at point $(0,0)$ to the left of $\mathcal{F}_d^{2}$ and end at point $(M,N)$ to the right of $\mathcal{F}_d^{2}$ is equal to the double number of unrestricted paths from $(0,0)$ to $(M,N)$.
\begin{proof}
 Similar to Lemma \ref{l4}.
\end{proof}
\end{lemma}
\begin{lemma}
\label{l8}
The weighted number of lattice paths from $(0,0)$  to $(M,N)$  with steps from $\mathbb{S}$ and filter restriction $\mathcal{F}_{-d}^{2}$ 
is
$$
Z(L_N((0,0)\to (M,N)\;|\;\mathcal{F}_{-d}^{2}))={N\choose\frac{N-M}{2}}+ {N\choose\frac{N-M}{2}+d},\;\;\text{ for}\;\; M\geq -d
$$ 
In other words, the weighted number of  paths that start at point $(0,0)$ to the right of $\mathcal{F}_{-d}^{2}$ and end at  point $(M,N)$ to the right of $\mathcal{F}_{-d}^{2}$ is equal to the number of unrestricted paths from $(0,0)$ to $(M,N)$ and from $(-2d,0)$ to $(M,N)$.
\begin{proof}
Similar to Lemma \ref{l5}.
\end{proof}
\end{lemma}

\section{Counting paths  with wall and one filter restriction}
\label{wallfilter}
In this section we will give a formula for the number of weighted paths with the left wall restriction $\mathcal{W}_0^L$ located at $x=0$ and the type $1$ filter restriction $\mathcal{F}_{d}^1$ located at $d=l-1$. Theorem \ref{TH4} gives the number of paths that  end to the left of the filter and Theorem \ref{TH2} gives the number of paths that end to the right of the filter. The resulting formula can be easily generalized to the arbitrary location of the wall and the filter.

Let us  denote by  $$L_N((0,0)\to (M,N)\;|\;\mathcal{W}_0^L,\mathcal{F}_{l-1}^{1})$$ 
the set of paths on $\mathcal{L}$ that start at $(0,0)$ and end at $(M,N)$  in the presence of the wall $\mathcal{W}_0^L$ and the filter $\mathcal{F}_{l-1}^1$. 
We will denote by $F_M^{(N)}$
the number of paths from $(0,0)$ to $(M,N)$ with the wall restriction $\mathcal{W}_0^L$ located at $x=0$. Due to  Lemma \ref{l2}:
\begin{equation*}
F_M^{(N)}= {N\choose\frac{N-M}{2}}-{N\choose{\frac{N-M}{2}-1}}.
\end{equation*}

\begin{theorem}\label{TH4}
The number of lattice paths from $(0,0)$ to $(M,N)$  with steps from $\mathbb{S}$ and with the wall restriction $\mathcal{W}_0^L$ and the filter restriction $\mathcal{F}_{l-1}^1$ when $0\leq M\leq l-2$ is given by
\begin{equation}\label{desire1}
|L_N((0,0)\to (M,N)\;|\;\mathcal{W}_0^L,\mathcal{F}_{l-1}^{1})|=F^{(N)}_M+\sum_{k=1}^{\bigl[\frac{N}{2l}+\frac{1}{2}\bigl]} F^{(N)}_{M-2kl}+\sum_{k=1}^{\bigl[\frac{N}{2l}\bigl]} F^{(N)}_{M+2kl}.
\end{equation}
\end{theorem}
\begin{proof}
The paths that we are going to count start to the left and end to the left of the filter $\mathcal{F}_{l-1}^{1}$. In this case filter  restriction at $x=l-1$ is equivalent to the wall restriction at  $x=l-2$ according to Lemma \ref{l3}, so $|L_N((0,0)\to (M,N)\;|\;\mathcal{W}_0^L,\mathcal{F}_{l-1}^{1})|=|L_N((0,0)\to (M,N)\;|\;\mathcal{W}_0^L,\mathcal{W}_{l-2}^R| $. Therefore  we can use the  Gessel-Zeilberger reflection principle \cite{GZ}. The proof of this result is presented in \cite{Kratt}. For readers convenience  we will repeat the proof of this fundamental result using slightly different notations below.
 
Let us denote by $W$ the group $W=\{w_1,w_2
\;|\;w_1^2=w_2^2=id\}$, generated by $w_1$ (the reflection w.r.t. $x=-1$) and $w_2$ (the reflection w.r.t. $x=l-1$)and  by  $\mathcal{H}=\{x=-1,x=l-1\}$ the set of reflection axes\footnote{$W$ is the Weyl group of $\hat{\mathfrak{sl}}_2$, the affine Kac-Moody algebra corresponding to $\mathfrak{sl}_2$}. The assignment $\mathrm{sgn}(w_1)=\mathrm{sgn}(w_2)=-1$ defines a signature character on $W$, $\mathrm{sgn}(w)=\mathrm{sgn}(w_1)^{n_1}\mathrm{sgn}(w_2)^{n_2}$, where $n_i$ is the number of occurrences of $w_i$ in $W$. It is easy to check that $\mathrm{sgn}(w)$ does not depend on the decomposition of $w$ in the product of generators and that $\mathrm{sgn}(w)(w^\prime)=\mathrm{sgn}(w)\mathrm{sgn}(w^\prime)$. 
 
We will denote the set of paths from point $A$ to point $B$ that stay between axes $\mathcal{H}$ as the set of "good" paths:  $$L_N^g(A\to B).$$

We will denote  the paths from point $A$ to point $B$ that visit any of the axes $\mathcal{H}$ as the set of "bad" paths: 
$$L_N^b(A\to B).$$ 
It is clear that to count the good paths from $(0,0)$ to $(M,N)$ we can remove bad paths from the set of unrestricted paths:
\begin{equation}
\label{xc}
    |L_N^g((0,0)\to (M,N))|=|L_N((0,0)\to (M,N))|- |L_N^b((0,0)\to (M,N))|.
\end{equation}
To obtain expression for $|L_N^b((0,0)\to (M,N))|$ we will firstly show that
\begin{equation}
\label{zero}
   \sum_{w\in W}(\mathrm{sgn}(w))|L_N^b(w(0,0)\to (M,N))|=0.
\end{equation}
 Consider a typical "bad" walk $\mathcal{P}$ from  $w(0,0)$ to $(M,N)$ that visits the axis $h\in\mathcal{H}$. We can pair this walk to  the walk $\tilde{\mathcal{P}}$ from  $w_hw(0,0)$ to $(M,N)$ obtained by reflecting the portion of $\mathcal{P}$ until its last visit of $h$ (see Figure \ref{act} for an example, $w=w_2,\;w_h=w_1$). It is clear that the pairing of walks is an involution and $\mathrm{sgn} w_i=-\mathrm{sgn} (w_hw_i)$.  All the terms in (\ref{zero}) can be arranged in such pairs. They cancel each other, and therefore the sum is zero.

If $w\neq id$, every walk starting from $w(0,0)$ and ending at $(M,N)$ is a "bad" walk, so for $w\neq id$ the set of unrestricted paths from $w(0,0)$ to $(M,N)$ is the set of bad paths from $w(0,0)$ to $(M,N)$:
\begin{equation}
\label{wbad}
    L_N(w(0,0)\to (M,N))=L_N^b(w(0,0)\to (M,N)),\;\;w\neq \;id
\end{equation}
 Now substituting (\ref{zero}), (\ref{wbad}) to (\ref{xc}) and denoting $L_N(w)=L_N(w(0,0)\to (M,N))$ we obtain
\begin{equation}
\label{enum}
|L_N^g(w(0,0)\to (M,N))|=\sum_{w\in W}(\mathrm{sgn}(w))|L_N(w)|.
\end{equation}
Note that
$
|L_N(w)|={N\choose\frac{N-M+x(w)}{2}},
$
where $x(w)$ is the $x$ coordinate of $w(0,0)$, as in Figure \ref{act}. 
\begin{figure}[h!]
	\centerline{\includegraphics[width=450pt]{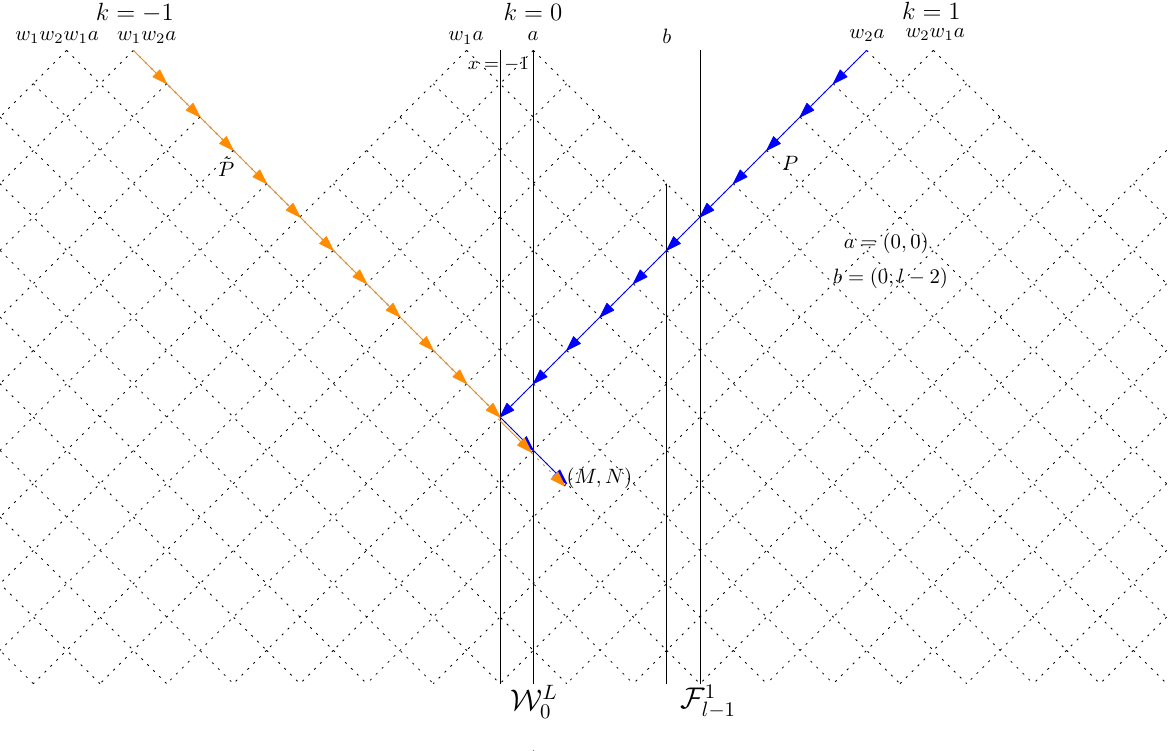}}
	\caption{
	Counting weighted paths that start at $(0,0)$ and end at $(M,N)$ for $M< l-1$ with the wall restriction $\mathcal{W}_0^L$ and the filter restriction $\mathcal{F}_{l-1}^1$. The reflection group $W$ acts on paths, touching axes of $\mathcal{H}$. Path $P$ and path $\tilde{P}$ after involution.}
	\label{act}
\end{figure}

To complete the proof we have to find the range of possible $w$ that contribute to the sum (\ref{enum}). For convenience, we will group elements $w\in W$ by pairs $\{w^L(k),w^R(k)\}$ with $k\in\mathbb{Z}$ and $x(w^R(k))-x(w^L(k))=2$. Note, that $\mathrm{sgn}w^R(k)=1$ and $\mathrm{sgn}w^L(k)=-1$. Precisely, we will have
\begin{equation}
\text{for $x>0$; $k=1,2,\ldots$: $w^R(k)=(w_2w_1)^k$ and $w^L(k)=w_2(w_1w_2)^{k-1}$},  \nonumber
\end{equation}
\begin{equation}
\text{for $x\leq0$; $k=0,-1,\ldots$: $w^R(k)=(w_1w_2)^{-k}$ and $w^L(k)=w_1(w_2w_1)^{-k}$}  .\nonumber
\end{equation}
 To find the possible range of $k$ for any given $(M,N)$ one can note that $x$ coordinate of $w^R(k)(0,0)$ is $2kl$. If we place endpoint $(M,N)$ on the right border $x=l-2$, the last path that could reach it must start at point $(N+l-2,0)$, which we set to be $w_L(k_{max})(0,0)$. This point belongs to the pair $\{w_L(k_{max}),w_R(k_{max})\}$. Consequently $w_R(k_{max})=N+l$ and, therefore, 
\begin{equation}
k_{max}=\Big[\frac{N+l}{2l}\Big]=\Big[\frac{N}{2l}+\frac{1}{2}\Big]. \nonumber
\end{equation} 
Similarly, 
\begin{equation}
k_{min}=-\Big[\frac{N}{2l}\Big]. \nonumber
\end{equation}
Combining all of the above, for the number of paths  in $L_N^g(w(0,0)\to (M,N))$ we obtain from (\ref{act}):
\begin{eqnarray}
\label{x}
|L_N^g(w(0,0)\to (M,N))|
=|L_N(id)|-|L_N(w_1)|+\sum_{k=1}^{k_{max}}\big(|L_N((w_2w_1)^k)|-|L_N(w_2(w_1w_2)^{k-1})|\big)+ \nonumber \\
+\sum_{j=1}^{j_{max}}\big(|L_N((w_1w_2)^j)|-|L_N(w_1(w_2w_1)^{j})|\big)={N\choose\frac{N-M}{2}}-{N\choose\frac{N-M}{2}-1}+ \nonumber \\
+\sum_{k=1}^{k_{max}}\Big({N\choose\frac{N-M}{2}+kl}-{N\choose\frac{N-M}{2}+kl-1}\Big) +\sum_{j=1}^{-k_{min}}\Big({N\choose\frac{N-M}{2}-kl}-{N\choose\frac{N-M}{2}-kl-1}\Big), \nonumber
\end{eqnarray}
that is
\begin{equation}
\label{f1r}
|L_N((0,0)\to (M,N);\mathbb{S}\;|\;\mathcal{W}_0^L,\mathcal{F}_{l-1}^{1},0\leq M\leq l-2)|=F^{(N)}_M+\sum_{k=1}^{\bigl[\frac{N}{2l}+\frac{1}{2}\bigl]} F^{(N)}_{M-2kl}+\sum_{k=1}^{\bigl[\frac{N}{2l}\bigl]} F^{(N)}_{M+2kl}.
\end{equation}
\end{proof}
\begin{theorem}\label{TH2}
The weighted number of lattice paths from $(0,0)$ to $(M,N)$  with steps from $\mathbb{S}$ and with the wall restriction $\mathcal{W}_0^L$ and the filter restriction $\mathcal{F}_{l-1}^1$ when $M> l-2$ is given by
\begin{equation}\label{desire2}
Z(L_N((0,0)\to (M,N)\;|\;\mathcal{W}_0^L,\mathcal{F}_{l-1}^{1}))=F^{(N)}_M+\sum_{k=1}^{\bigl[\frac{N-l+1}{2l}\bigl]} F^{(N)}_{M+2kl}.
\end{equation}
\end{theorem}
\begin{proof}
To count the weight of "good" paths (paths from $(0,0)$ to $(M,N)$ that obey the restrictions of the theorem) we can subtract the weight  of "bad" paths from the weight of unrestricted paths:
\begin{equation}
\label{}
    Z(L_N^g((0,0)\to (M,N)))=Z(L_N((0,0)\to (M,N)))- Z(L_N^b((0,0)\to (M,N))).
\end{equation}
Here $L_N^g((0,0)\to (M,N))$ are "good" paths and $L_N^b((0,0)\to (M,N))$ are "bad" paths. According to the assumptions of the theorem the "bad" paths are the ones that visit the axis $x=-1$ and the ones that have a step $(d,y)\to(d-1,y+1)$(this step is forbidden by the filter $\mathcal{F}_{l-1}^{1}$). But when we transition from the weighted paths to non-weighted paths by a map $\psi$ according to Lemma \ref{l4}, the step  $(d,y)\to (d-1,y+1)$ is created  as a result of the action of $\psi$ on a weighted path that has a step $(d+1,y+1)\xrightarrow[]{2}(d,y+2)$. Therefore  for non-weighted paths ending to the right of $\mathcal{F}_{l-1}^{1}$ there are no forbidden steps and the only "bad" paths are the ones that touch the axis $x=-1$. 

According to the reflection principle, to any such "bad" path $\mathcal{P}$ from $(0,0)$ to $(M,N)$ there corresponds a partially reflected path $\mathcal{P}^\prime$ from $(-2,0)$ to $(M,N)$. Excluding these paths gives us $ Z(L_N^g((0,0)\to (M,N)))=|L_N((0,0)\to (M,N)))|- |L_N((-2,0)\to (M,N))|$. Such procedure gives us an exact expression for $ Z(L_N^g((0,0)\to (M,N)))$ for $N<3l-1$.

But as  $N=3l-1$  a path visiting the axis $x=-1$ may actually be a "good" path. Such path, for example $\mathcal{P}_0$ (see Figure \ref{WF1R}), can be obtained after an action of $\psi$  (see Lemma \ref{l4}) on a "good" weighted path $\mathcal{P}$ as $\psi(\mathcal{P})=\{\mathcal{P}_{0},\mathcal{P}^\prime\}$. The path $\mathcal{P}_0$ has been reflected to $\mathcal{P}_{-2}$, and has then been excluded by subtraction of $|L_N((-2,0)\to (M,N))|$ from the total number of paths. To release from this contradiction we can notice that the path $\mathcal{P}_{-2}$ visits the axis $x=-l-1$. Therefore we can construct a path $\mathcal{P}_{-2l}$ by reflecting a portion of $\mathcal{P}_{-2}$ before the last visit of $x=-l-1$ (see Figure \ref{WF1R}). We then want to include the path $\mathcal{P}_{-2l}$ to compensate the exclusion of $\mathcal{P}_{0}$ by adding $|L_N((-2l,0)\to (M,N))|$, so $ Z(L_N^g((0,0)\to (M,N)))=|L_N((0,0)\to (l,N)))|- |L_N((-2,0)\to (M,N))|+ |L_N((-2l,0)\to (M,N))|$ for $N=3l-1$. 

\begin{figure}[h!]
	\centerline{\includegraphics[width=450pt]{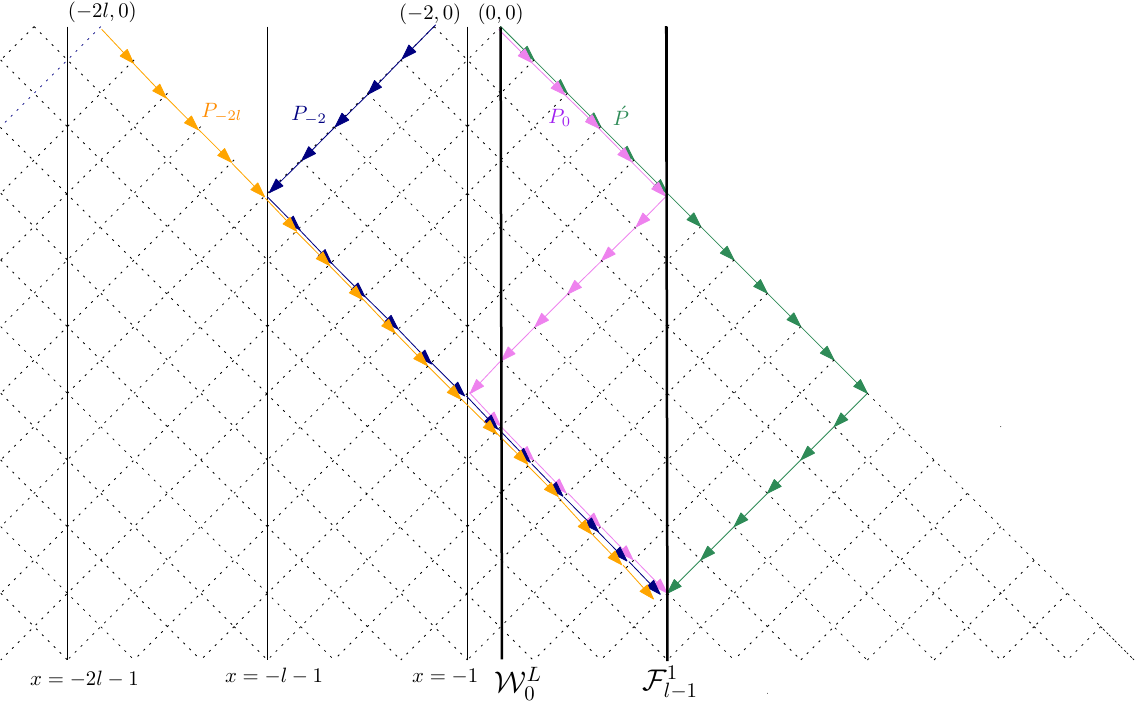}}
	\caption{Counting weighted paths that start at $(0,0)$ and end at $(M,N)$ for $M\geq l-1$ with the wall restriction $\mathcal{W}_0^L$ and the filter restriction $\mathcal{F}_{l-1}^1$} via unweighted walks counting.
	\label{WF1R}
\end{figure}

Further,  as $N>3l-1$ we will need to exclude "bad" paths that visited the axis $x=-1$ once before the action of $\psi$, but visit this axis twice after the action of $\psi$. Note, that they were excluded by subtraction of $|L_N((-2,0)\to (M,N))|$, but included by addition of $|L_N((-2l,0)\to (M,N))|$ . It is clear that the exclusion of such paths will lead to  subtraction of $|L_N((-2l-2,0)\to (M,N))|$.

Continuing this procedure, overall we get
\begin{equation}\label{zzzz}
Z(L_N^g(w(0,0)\to (M,N)))=\sum_{w\in \tilde{W}}\mathrm{sgn}(w)|L_N(w(0,0)\to (M,N))|,
\end{equation}
where $\tilde{W}$ is a subset of  elements $w\in W$, such that $x(w)\leq 0$ (the corresponding reflection axes lie at $x\leq 0$). In terms of $F_M^{(N)}$ it is
\begin{equation*}
Z(L_N^g(w(0,0)\to (M,N)))=F^{(N)}_M+\sum_{k=1}^{\bigl[\frac{N-l+1}{2l}\bigl]} F^{(N)}_{M+2kl}.
\end{equation*}

\end{proof}
\textit{Note}: The number of paths (\ref{desire1}) starting at $(0,0)$ and staying within the first strip and the number of paths (\ref{desire2}) that start at $(0,0)$ and end to the right of the filter differ by $\sum_{k=1}^{\bigl[\frac{N}{2l}+\frac{1}{2}\bigl]} F^{(N)}_{M-2kl}$. This is an  alternating  sum of the number of paths in the sets of unrestricted paths from $(2kl,0)$ and $(2kl-2,0)$, which for $k=1,2,\dots$ start at $x>0$. Cancellation of this sum corresponds to creation of the step $(l-1,y)\to (l-2,y+1)$  as a result of the action of $\psi$ on a path that has a step $(l,y+1)\xrightarrow[]{2}(l-1,y+2)$(see  Lemma \ref{l4}).The summation limit is determined following the same procedure as in Theorem \ref{TH4}.

\section{Counting paths with two filter restrictions}
\label{twofilters}
Let us consider the set of paths on $\mathcal{L}$ that start at $(0,0)$ and end at $(M,N)$  in the presence of the filter $\mathcal{F}_{l-1}^1$ and the filter $\mathcal{F}_{2l-1}^2$. We will denote it by $$L_N((0,0)\to (M,N)\;|\;\mathcal{F}_{l-1}^{1}, \mathcal{F}_{2l-1}^2).$$  
\begin{theorem}\label{TH3}
The weighted number of lattice paths from $(0,0)$ to $(M,N)$   with steps from $\mathbb{S}$ and with the filter restriction $\mathcal{F}_{l-1}^1$  and the filter restriction $\mathcal{F}_{2l-1}^2$ when $l-1\leq M< 2l-1$ is given by
\begin{equation}\label{th3}
Z(L_N((0,0)\to (M,N)\;|\;\mathcal{F}_{l-1}^1,\mathcal{F}_{2l-1}^{2}))=\sum_{k=0}^{\bigl[\frac{N-l+1}{2l}\bigl]} (-1)^{k}C^{(N)}_{M+2kl}+\sum_{k=2}^{\bigl[\frac{N}{2l}+1\bigl]} (-1)^{k-1}C^{(N)}_{M-2kl+2},
\end{equation}
where $C^{(N)}_M={N\choose\frac{N-M}{2}}$  is the number of unrestricted paths from $(0,0)$ to $(M,N)$.
\end{theorem}
\begin{proof}

Let us denote by $L_N^g((0,0)\to (M,N))$  the "good" paths: paths from $(0,0)$ to $(M,N)$ that obey the restrictions of the theorem and $L_N^b((0,0)\to (M,N))$ the "bad" paths: paths that touch the axis $x=2l-1$. To count the weight of "good" paths  we can subtract the weight  of "bad" paths from the weight of all paths:
\begin{equation}
\label{zg3}
    Z(L_N^g((0,0)\to (M,N)))=Z(L_N((0,0)\to (M,N)))- Z(L_N^b((0,0)\to (M,N))).
\end{equation}
We will now express this in terms of unrestricted paths. Firstly, we must note that due to Lemma \ref{l4} we have $Z(L_N((0,0)\to (M,N)))=|L_N((0,0)\to (M,N)))|$. The filter  restriction $\mathcal{F}_{2l-1}^{2}$ at $x=2l-1$ is equivalent to the wall restriction $\mathcal{W}_{2l-2}^R$ at $x=2l-2$ (see Lemma \ref{l6}). Below we will illustrate the counting of the weight of "bad" paths $Z(L_N^b((0,0)\to (M,N)))$ for values of $N$ from $l-1$ to $5l-1$ (see Figure \ref{figth33}).\\

\begin{figure}[h!]
	\centerline{\includegraphics[width=450pt]{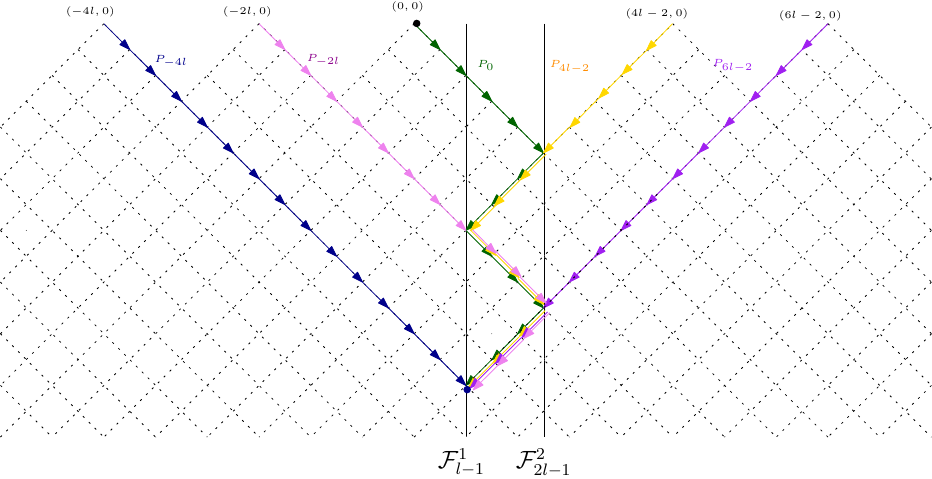}
	}
	\caption{Counting weighted walks from $(0,0)$ to $(M,N)$ for $l-1\leq M< 2l-1$ in the presence of filters $\mathcal{F}_{l-1}^{1}$ and $\mathcal{F}_{2l-1}^{2}$ in terms of unweighted walks counting.
	}
	\label{figth33}
\end{figure}

For $l-1\leq N<2l-1$ there are no "bad" paths, since the endpoint $(l,N)$ has not reached the axis $x=2l-1$.

For $ 2l-1\leq N<3l-1$ all "bad" paths from $(0,0)$ to $(M,N)$ have weight 1, so $Z(L_N^b((0,0)\to (M,N)))=|L_N^b((0,0)\to (M,N))|$. 
According to the reflection principle, to any "bad" path $\mathcal{P}_0$ from $(0,0)$ to $(M,N)$ there corresponds a path $\mathcal{P}_{4l-1}$ from $(4l-1,0)$ to $(M,N)$, so  $Z(L_N^b((0,0)\to (M,N))))=|L_N((4l-1,0)\to (M,N)))|$.

As $3l-1\leq N<4l-1$ the "bad" paths may have the step $(l,y+1)\xrightarrow[]{2}(l-1,y+2)$ and therefore have the weight 2. Therefore we have to exclude two unweighted paths to compensate the weight of the "bad" path $\mathcal{P}_0$. These paths are $\mathcal{P}_{4l-1}$ and $\mathcal{P}_{-2l}$, which is obtained from $\mathcal{P}_{4l-1}$ by reflecting the its portion until the first visit of $x=l-1$. 
Therefore we have $Z(L_N^b((0,0)\to (M,N)))=|L_N((4l-1,0)\to (M,N))|+ |L_N((-2l,0)\to (M,N))|$ for $3l-1\leq N<4l-1$. 

The path $\mathcal{P}_{-2l}$ coincides with $\mathcal{P}_0$ after its visit of $x=l-1$. But since it is forbidden for the initial path $\mathcal{P}_0$ to cross the axis $x=2l-1$ , it is also forbidden to the reflected path $\mathcal{P}_{-2l}$ to cross $x=2l-1$. Therefore we must exclude from $L_N((4l-1,0)\to (M,N))$ the paths touching the axis $x=2l-1$. We can again apply the reflection principle w.r.t. $x=2l-1$. To a path $\mathcal{P}_{-2l}$ corresponds a path $\mathcal{P}_{6l-1}$ from $(6l-1,0)$ to $(M,N)$. Therefore we have $Z(L_N^b((0,0)\to (M,N)))=|L_N((4l-1,0)\to (M,N))|+ |L_N((-2l,0)\to (M,N))|-|L_N((6l-1,0)\to (M,N))|$ for $4l-1<N<5l-1$. 

As $5l-1\leq N<6l-1$ the paths with weight 4 may appear. Expressing the number of weighted paths in terms of unweighted paths, similar to the above we get the path $\mathcal{P}_{-4l}$ from  by reflecting $\mathcal{P}_{6l-1}$ w.r.t $x=l-1$. Including this path gives us $Z(L_N^b((0,0)\to (M,N)))=|L_N((4l-1,0)\to (M,N))|+ |L_N((-2l,0)\to (M,N))|-|L_N((6l-1,0)\to (M,N))|+ |L_N((-4l,0)\to (M,N))|$. 

Overall, continuing this procedure of counting "bad" paths and substituting their number to (\ref{zg3}) we get
\begin{equation*}
  Z(L_N^g((0,0)\to (M,N)))= \sum_{w\in \hat{W}}(\mathrm{sgn}w)|L_N(w(0,0)\to (M,N))|,
\end{equation*}
where $\hat{W}$ is the group generated by ${\tilde{\psi},\phi}$. Here $\tilde{\psi}$ is the reflection w.r.t. the axis $x=2l-1$ and $\mathrm{sgn}\psi=1$ and $\phi$ is the reflection with respect to the axis $x=l-1$ and $\mathrm{sgn}\phi=-1$. Since
$
|L_N(w(0,0)\to (M,N))|={N\choose\frac{N-M+x(w)}{2}},
$
where $x(w)$ is the $x$ coordinate of $w(0,0)$, as in Figure \ref{figth3}. Expressing this in terms of $C_M^{(N)}$ we get (\ref{th3}). 

 Now we will find the upper limits of both sums in (\ref{th3}).
 Consider firstly the second sum in (\ref{th3}) which is the contribution from the paths that start to the right of $(0,0)$.  Their starting points are $(2kl-2,0), \;\;k=2,3\dots$. If we place endpoint $(M,N)$ at $M=2l-2$, the last path that could reach it from the right must start at the point $(N+2l-2,0)$. So for the upper limit $k_{max}$ of the second sum  we have $2k_{max}l-2=N+2l-2$, and therefore 
\begin{equation}
k_{max}=\Big[\frac{N}{2l}+1\Big]. \nonumber
\end{equation} 
Similarly, consider the first sum in (\ref{th3}) which is the contribution from the paths that start to the left of $(0,0)$. Their starting points are $(-2\tilde{k}l,0),
\;\; \tilde{k}=0,1,\dots $. If we place endpoint $(M,N)$ at $M=l-1$, the last path that could reach it from the left must start at point $(-N+l-1,0)$. So for the upper limit $\tilde{k}_{max}$ of the first sum we have $2\tilde{k}_{max}l=N-l+1$, therefore
\begin{equation}
\tilde{k}_{max}=\Big[\frac{N-l+1}{2l}\Big]. \nonumber
\end{equation} 
\end{proof}
Due to translation invariance we can generalize (\ref{th3}):
\begin{remark}\label{REM}
The weighted number of lattice paths from $(-2Al,0)$ to $(M,N)$  for $A\geq 0$  with steps from $\mathbb{S}$   and with the filter restriction $\mathcal{F}_{l-1}^1$  and the filter restriction $\mathcal{F}_{2l-1}^2$ when $l-1\leq M< 2l-1$ is given by
\begin{eqnarray}\label{th32}
Z(L_N((-2Al,0)\to (M,N)\;|\;\mathcal{F}_{l-1}^1,\mathcal{F}_{2l-1}^{2}))=&\sum\limits_{k=A}^{\bigl[\frac{N-l+1}{2l}\bigl]} (-1)^{k-A}C^{(N)}_{M+2kl}+\\
+&\sum\limits_{k=A}^{\bigl[\frac{N}{2l}-1\bigl]} (-1)^{k-A+1}C^{(N)}_{M-2(k+2)l+2},\nonumber
\end{eqnarray}
The weighted number of lattice paths  from $(-2Bl-2,0)$ to $(M,N)$  for $B\geq 0$  with steps from $\mathbb{S}$ and with the filter restriction $\mathcal{F}_{l-1}^1$  and the filter restriction $\mathcal{F}_{2l-1}^2$ when $l-1\leq M< 2l-1$ is given by
\begin{eqnarray}\label{th33}
Z(L_N((-2Bl-2,0)\to (M,N)\;|\;\mathcal{F}_{l-1}^1,\mathcal{F}_{2l-1}^{2}))=&\sum\limits_{k=B}^{\bigl[\frac{N-l-1}{2l}\bigl]} (-1)^{k-B}C^{(N)}_{M+2kl-2}\\
+&\sum\limits_{k=B}^{\bigl[\frac{N-2}{2l}-1\bigl]}(-1)^{k-B+1}C^{(N)}_{M-2(k+2)l},\nonumber
\end{eqnarray}
where $C^{(N)}_M={N\choose\frac{N-M}{2}}$  is the number of unrestricted paths from $(0,0)$ to $(M,N)$.
\end{remark}

\section{Counting paths with wall and two filter restrictions}
\label{walltwofilters}
Let us consider set of paths on $\mathcal{L}$ that start at $(0,0)$ and end at $(M,N)$  in the presence of the  wall $\mathcal{W}_0^L$ the  two filters $\mathcal{F}_{l-1}^1$ and $\mathcal{F}_{2l-1}^2$. We will denote it by $$L_N((0,0)\to (M,N)\;|\;\mathcal{W}_0^L,\mathcal{F}_{l-1}^{1}, \mathcal{F}_{2l-1}^2).$$  
\begin{theorem}\label{TH5}
The weighted number of lattice paths from $(0,0)$ to $(M,N)$  with steps from $\mathbb{S}$ and with the  wall $\mathcal{W}_0^L$ and filter restrictions $\mathcal{F}_{l-1}^1$  and $\mathcal{F}_{2l-1}^2$ when $l-1\leq M< 2l-1$ is given by
\begin{equation}\label{th4}
Z(L_N((0,0)\to (M,N)\;|\;\mathcal{W}_0^L, \mathcal{F}_{l-1}^1,\mathcal{F}_{2l-1}^{2}))=
\sum_{k=0}^{\bigl[\frac{N-l+1}{4l}\bigl]} F^{(N)}_{M+4kl}+\sum_{k=0}^{\bigl[\frac{N-2l}{4l}\bigl]}F^{(N)}_{M-4kl-4l},
\end{equation}
where $F_M^{(N)}= {N\choose\frac{N-M}{2}}-{N\choose{\frac{N-M}{2}-1}}$.
\end{theorem}
\begin{proof}
We will start by using results of Theorem \ref{TH2} as the boundary conditions on the axis $x=l-1$ and then apply Theorem \ref{TH3}. 

Let us consider the path restrictions of Theorem \ref{TH2} (the wall $\mathcal{W}_0^L$ and  filter $\mathcal{F}_{l-1}^1$). According to (\ref{zzzz}) the weighted number of paths that end $(l-1,N)$ is
\begin{eqnarray}\label{qgg}
Z(L_N^g(w(0,0)\to (l-1,N)))=\sum_{w: x(w(0,0))\leq 0}(\mathrm{sgn}(w))|L_N(w(0,0)\to (l-1,N))|=\\
=\sum_{i=0}^{\ldots}|L_N((w_1w_2)^i(0,0)\to (l-1,N))|-
\sum_{i=0}^{\ldots}|L_N(w_1(w_1w_2)^i(0,0)\to (l-1,N))|.\nonumber
\end{eqnarray}
Here and below by $\sum^{\dots}$ we denote the summation over all possible $i$ that give nonzero contribution. These boundary conditions on the axis $x=l-1$ determine the number  of paths that end to the right of the filter $\mathcal{F}_{l-1}^1$. Therefore  we may change conditions to the left of the filter as long as the weighted number of paths that end at $(l-1,N)$ remains the same. Expression (\ref{qgg}) and Lemma \ref{l4} suggest that 
we can consider two separate sets  of initial conditions with no wall restriction. The first set of initial conditions  includes $(w_1w_2)^i(0,0)$ as the path starting points while only the filter $\mathcal{F}_{l-1}^1$ restriction is present.  The second  set of initial conditions  includes $w_1(w_1w_2)^i(0,0)$ the path starting points and only the filter restriction $\mathcal{F}_{l-1}^1$.
\begin{figure}[h!]
\label{F12}
	\centerline{\includegraphics[width=450pt]{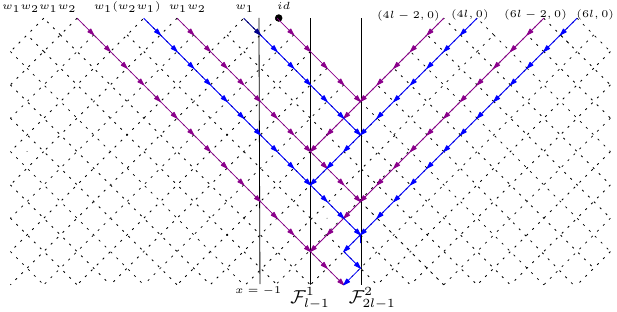}
}
	\caption{Counting weighted walks from $(0,0)$ to $(M,N)$ for $l-1\leq M< 2l-1$ in the presence of  the wall $\mathcal{W}_0^L$ and filters $\mathcal{F}_{l-1}^{1}$ and $\mathcal{F}_{2l-1}^{2}$ in terms of unweighted walks counting. The starting points are labeled by the corresponding reflection group elements acting on $(0,0)$. The violet paths contribute to $Z_+^0$ and the blue paths to $Z_-^0$. It is clear that the path starting from $w_1w_2$ contributes to both $Z_+^0$ and $Z_+^1$.The contribution of this path is canceled in $Z_+^0+Z_+^1$.
 }
\end{figure}

To prove (\ref{th4}) we will take these two sets of initial conditions with $\mathcal{F}_{l-1}^1$  and place  another filter restriction $\mathcal{F}_{2l-1}^2$. Therefore the problem is reduced to two separate problems. The first problem is to count weighted number  $Z_+^i$ of paths  from  each $(w_1w_2)^i(0,0)$ to $(M,N)$ in the presence of $\mathcal{F}_{l-1}^1$  and  $\mathcal{F}_{2l-1}^2$ and to  sum these numbers over $i$. The second problem is to count weighted number  $Z_-^i$ of paths  from all $w_1(w_1w_2)^i(0,0)$ to $(M,N)$ in the presence of $\mathcal{F}_{l-1}^1$  and  $\mathcal{F}_{2l-1}^2$  and to  sum these numbers over $i$. The required weighted number of paths will be given by
\begin{equation}
\label{z12}
Z(L_N((0,0)\to (M,N)\;|\;\mathcal{W}_0^L, \mathcal{F}_{l-1}^1,\mathcal{F}_{2l-1}^{2}))=\sum_{i=0}^{\dots}Z_+^i-\sum_{i=0}^{\dots}Z_-^i.
\end{equation}
We will further specify the summation limits in (\ref{z1}) and (\ref{z2}).  

 To obtain expression for $Z_+^i$ and $Z_-^i$ one can use Theorem \ref{TH3}  and Remark \ref{REM}.   Note that $(w_1w_2)^i(0,0)=(-2il,0)$  and $w_1(w_1w_2)^i(0,0)=(-2il-2,0)$ for $i=0,1,\dots $.  So we can apply Remark \ref{REM} for $A=i$ and $B=i$ to calculate $Z_+^i$ and $Z_-^i$ respectively:
 \begin{eqnarray*}
  Z_+^i=\sum_{k=i}^{\bigl[\frac{N-l+1}{2l}\bigl]} (-1)^{k-i}C^{(N)}_{M+2kl}+\sum_{k=i}^{\bigl[\frac{N}{2l}-1\bigl]} (-1)^{k-i+1}C^{(N)}_{M-2(k+2)l+2},\\
     Z_-^i=\sum_{k=i}^{\bigl[\frac{N-l-1}{2l}\bigl]} (-1)^{k-i}C^{(N)}_{M+2kl-2}+\sum_{k=i}^{\bigl[\frac{N-2}{2l}-1\bigl]} (-1)^{k-i+1}C^{(N)}_{M-2(k+2)l}.
 \end{eqnarray*}
 
Firstly, we must calculate  contributions to (\ref{z12}) that come from the sum of $Z_+^i$. Most of the terms in this sum will cancel each other. Below we will illustrate this for $Z_+^0$ and $Z_+^1$:
\begin{equation*}
Z_+^0=\sum_{k=0}^{\bigl[\frac{N-l+1}{2l}\bigl]}(-1)^kC^{(N)}_{M+2kl}+\sum_{k=0}^{\bigl[\frac{N}{2l}-1\bigl]} (-1)^{k-1}C^{(N)}_{M-2(k+2)l+2},
\end{equation*}
\begin{equation*}
Z_+^1=\sum_{k=1}^{\bigl[\frac{N-l+1}{2l}\bigl]}(-1)^{k-1}C^{(N)}_{M+2kl}+\sum_{k=1}^{\bigl[\frac{N}{2l}-1\bigl]} (-1)^{k}C^{(N)}_{M-2(k+2)l+2}.
\end{equation*}
When added together most of the terms cancel each other and we get $$Z_+^0 + Z_+^1=C_M^{(N)}-C_{M-4l+2}^{(N)}$$
In fact,  the similar cancellation of terms will take place for all such pairs of  $Z_+^{i}$  and $Z_+^{i+1}$ for even $i$. 
\begin{eqnarray*}
Z_+^i+Z_+^{i+1}=\sum_{k=i}^{\bigl[\frac{N-l+1}{2l}\bigl]}(-1)^{k-i}C^{(N)}_{M+2kl}+\sum_{k=i}^{\bigl[\frac{N}{2l}-1\bigl]} (-1)^{k-i+1}C^{(N)}_{M-2(k+2)l+2}
+\\
+\sum_{k=i+1}^{\bigl[\frac{N-l+1}{2l}\bigl]}(-1)^{k-i-1}C^{(N)}_{M+2kl}+\sum_{k=i+1}^{\bigl[\frac{N}{2l}-1\bigl]} (-1)^{k-i}C^{(N)}_{M-2(k+2)l+2}=C^{(N)}_{M+2il}+C^{(N)}_{M-2il-4l+2}.
\end{eqnarray*}
Denoting $i=2k$ we get
\begin{equation}
\label{z1}
\sum_{i=0}^{\dots}Z_+^i=\sum_{k=0}^{\bigl[\frac{N-l+1}{4l}\bigl]}C^{(N)}_{M+4kl}-\sum_{k=0}^{\bigl[\frac{N-2l}{4l}\bigl]}C^{(N)}_{M-4kl-4l+2}.
\end{equation}
 Contributions to (\ref{z12}) that come from the sum of   $Z_-^i$ are obtained following the same observation. Similarly, we get
\begin{equation}
\label{z2}
\sum_{i=0}^{\dots}Z_-^i=\sum_{k=0}^{\bigl[\frac{N-l-1}{4l}\bigl]}C^{(N)}_{M+4kl+2}- \sum_{k=0}^{\bigl[\frac{N-2l-2}{4l}\bigl]}C^{(N)}_{M-4kl-4l}.
\end{equation}
The summation limits are obtained in the same manner as in Theorem \ref{TH3}. For example, the second sum  in (\ref{z1}) is the contribution from the paths with starting points $(4(k+1)l-2,0), \;\;k=0,1\dots$. The last path that could reach $M=2l-2$ from the right starts at the point $(N+2l-2,0)$. So the upper limit of summation
$
k_{max}=\Big[\frac{N-2l}{4l}\Big]$.
Similarly, the first sum  in (\ref{z1}) is the contribution from the paths with starting points $(-4\tilde{k}l,0),
\;\; \tilde{k}=0,1,\dots $. The last path that could reach $M=l-1$ from the left starts at the point $(-N+l-1,0)$. So  the upper limit of summation 
$
\tilde{k}_{max}=\Big[\frac{N-l+1}{4l}\Big]
$. 
Substituting  (\ref{z1}) (\ref{z2}) to  (\ref{z12}) we obtain (\ref{th4}).
\end{proof}

\section{Counting paths with wall and multiple filter restrictions}
\label{multfilt}
In previous sections we have proven auxiliary theorems that provide enumerative formulas for counting paths in the presence of the wall and a number of filters less than three. We are now ready to prove the main theorem. Let us consider the set of paths on $\mathcal{L}$ that start at $(0,0)$ and end at $(M,N)$  with steps from $\mathbb{S}$ in the presence of the filter $\mathcal{F}_{l-1}^1$ of type 1 and the filters $\mathcal{F}_{nl-1}^2$ of type 2, where $n=2,3,\dots$. We will denote this set by 
\begin{equation}\label{mmm}
L_N((0,0)\to (M,N)\;|\;\mathcal{F}_{l-1}^1,\{\mathcal{F}_{(n+1)l-1}^{2}\}, n\in \mathbb{Z}_+).
\end{equation}
\begin{figure}[h!]
	\centerline{\includegraphics[width=350pt]
	{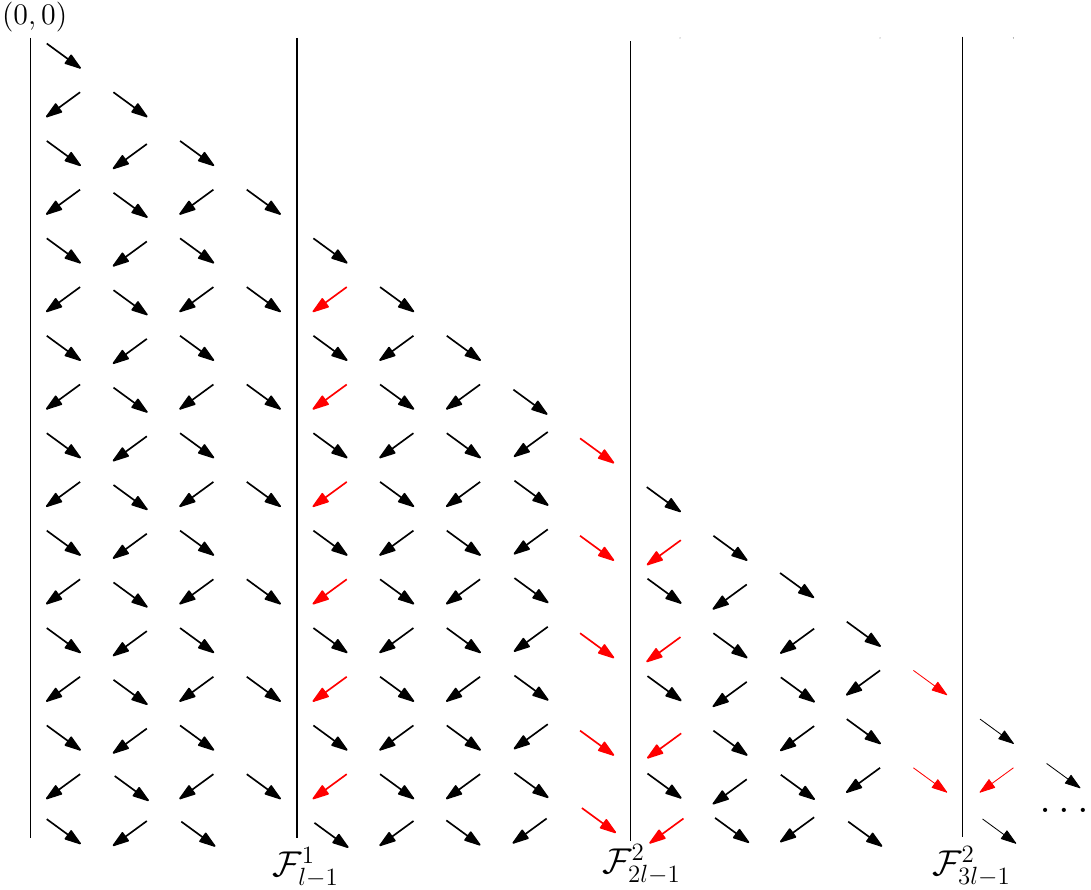}
}
	\caption{The arrangement of filters $\mathcal{F}_{l-1}^1$ and the filters $\mathcal{F}_{(j+1)l-1}^2$, where $j=1,2,\dots$. The $j$-th strip is located between $\mathcal{F}_{(j-1)l-1}^2$ and $\mathcal{F}_{jl-1}^2$.
}
\label{figth3}
\end{figure}
\begin{definition}
	  We will denote by multiplicity function in the $j$-th strip the weighted number of paths in set (\ref{mmm})  with the endpoint $(M,N)$ that lies within  $(j-1)l-1 \leq M< jl-2$
	\begin{equation}
	M^j_{(M,N)} = Z(L_N((0,0)\to (M,N)\;|\;\mathcal{F}_{l-1}^1,\{\mathcal{F}_{(n+1)l-1}^{2}\}, n\in \mathbb{Z}_+)),
	\end{equation}
	where $M\geq 0$, $j\geq 2$ and $j=\Big[\frac{M+1}{l}+1\Big]$. 
\end{definition}
\begin{theorem}
\label{tj}
	The multiplicity function in the $j$-th strip  is given by
	\begin{eqnarray*}
	\label{mj}
	M^{j}_{(M,N)} = 2^{j-2}\Big(\sum_{k=0}^{\big[\frac{N-(j-1)l+1}{4l}\big]}P_j(k) F^{(N)}_{M+4kl}+\sum_{k=0}^{\big[\frac{N-jl}{4l}\big]}P_j(k)F^{(N)}_{M-4kl-2jl} -\\-\sum_{k=0}^{\big[\frac{N-(j+1)l+1}{4l}\big]}Q_j(k)F^{(N)}_{M+2l+4kl}-\sum_{k=0}^{\big[\frac{N-jl-2l}{4l}\big]}Q_j(k)F^{(N)}_{M-4kl-2(j+1)l} \Big),
	\end{eqnarray*}
	where
	\begin{equation}
	\label{ppqq}
	P_j(k)=\sum_{i=0}^{\big[\frac{j}{2}\big]}\binom{j-2}{2i}\binom{k-i+j-2}{j-2},\;\;\;\;\;
	Q_j(k)=\sum_{i=0}^{\big[\frac{j}{2}\big]}\binom{j-2}{2i+1}\binom{k-i+j-2}{j-2},
	\end{equation}
	\begin{equation*}
F_M^{(N)}= {N\choose\frac{N-M}{2}}-{N\choose{\frac{N-M}{2}-1}}.
\end{equation*}
\end{theorem}
\begin{proof}
 
We will firstly use the expression (\ref{th4}) for the weighted number of paths  obtained in Theorem \ref{TH5} as the base of induction for the proof of (\ref{mj}). The multiplicity $M_{(M,N)}^2$ in the second strip is given by expression (\ref{th4}):

\begin{equation}\label{m2}
M_{(M,N)}^2(l)=Z(L_N((0,0)\to (M,N)\;|\;\mathcal{W}_0^L, \mathcal{F}_{l-1}^1,\mathcal{F}_{2l-1}^{2}))=
\sum_{k=0}^{\bigl[\frac{N-l+1}{4l}\bigl]} F^{(N)}_{M+4kl}+\sum_{k=0}^{\bigl[\frac{N-2l}{4l}\bigl]}F^{(N)}_{M-4kl-4l}.
\end{equation}
Clearly this expression satisfies (\ref{mj}) for $j=2$ with $P_2(k)=1,\;Q_2(k)=0$.

Expression (\ref{m2}) serves as the base of induction. For the inductive step we consider
	\begin{eqnarray*}
	M^j_{(M,N)} = 2^{j-2}\Big(\sum_{k=0}^{\big[\frac{N-(j-1)l+1}{4l}\big]}P_j(k)F^{(N)}_{M+4kl}+\sum_{k=0}^{\big[\frac{N-jl}{4l}\big]}P_j(k)F^{(N)}_{M-4kl-2jl} -\\-\sum_{k=0}^{\big[\frac{N-(j+1)l+1}{4l}\big]}Q_j(k)F^{(N)}_{M+2l+4kl}-\sum_{k=0}^{\big[\frac{N-(j+2)l}{4l}\big]}Q_j(k)F^{(N)}_{M-4kl-2(j+1)l} \Big),
	\end{eqnarray*}
 to hold for $j$-th strip, located between  $\mathcal{F}_{(j-1)l-1}^2$ and $\mathcal{F}_{jl-1}^2$, where $P_j(k)$ and $Q_j(k)$ given by (\ref{ppqq}). We will use the auxiliary theorems proven in the previous sections to get the multiplicity function for the $j+1$-th strip, which  is located between $\mathcal{F}_{jl-1}^2$ and $\mathcal{F}_{(j+1)l-1}^2$.  We will proceed the proof in 2 steps. 

At the first step we consider the setup with filters $\mathcal{F}_{l-1}^1,\mathcal{F}_{2l-1}^2, \dots , \mathcal{F}_{jl-1}^2$ but without  filter $\mathcal{F}_{(j+1)l-1}^2$. We will denote the weighted number of paths from $(0,0)$ to $(M,N)$  under these restrictions as 
\begin{equation}
	\tilde{M}^{j}_{(M,N)}(l) = Z(L_N((0,0)\to (M,N)\;|\;\mathcal{F}_{l-1}^1,\dots, \mathcal{F}_{jl-1}^{2}),
	\end{equation}
For  $(j-1)l-1\leq M< jl-1$  this function coincides with multiplicity function for $j$-th strip $M^j_{(M,N)}$ but for $M\geq jl-1$  the weighted number of paths is 
\begin{equation}
\label{kk}
	\tilde{M}^{j}_{(M,N)}(l) = 2^{j-1}\Big(\sum_{k=0}^{\big[\frac{N-jl+1}{4l}\big]}P_j(k)F^{(N)}_{M+4kl} -\sum_{k=0}^{\big[\frac{N-(j+2)l+1}{4l}\big]}Q_j(k)F^{(N)}_{M+2l+4kl} \Big),
\end{equation} 
 This expression  differs from $M^j_{(M,N)}$ by two sums which correspond to number of unrestricted paths in the sets  with the initial points  with $x>jl-1$ and by a factor $2$. It could be proven  similar to Theorem \ref{TH2} via unweighted path counting by creation of the step $(jl-1,y)\to (jl-2,y+1)$  as a result of the action of $\psi$ on a path that has a step $(jl,y+1)\xrightarrow[]{2}(jl-1,y+2)$ and doubling of the total number of paths when passing through $\ \mathcal{F}_{jl-1}^2$(see  Lemma \ref{l7}).

 At the second step  we consider  value of (\ref{kk}) on the axis $x=jl-1$ as the boundary conditions and use the same proof technique as in Theorem \ref{TH5}. The boundary conditions on the axis $x=jl-1$ determine the number of paths that end to the right of the filter $\mathcal{F}_{jl-1}^2$. Therefore we may change conditions to the left of the filter as long as the weighted number of paths that end at $(jl-1,N)$ remains the same.  
 It is clear that we can consider  $(w_1w_2)^i(0,0)$ and $w_1(w_1w_2)^i(0,0)$ as the path starting points:
 \begin{eqnarray*} 
	\tilde{M}^{j}_{(M,N)} = 2^{j-1}\Big(&\sum\limits_{k=0}^{\big[\frac{N-jl+1}{4l}\big]} P_j(k)\Big(&|L_N((w_1w_2)^{2k}(0,0)\to (jl-1,N))|-\\
	& &-|L_N(w_1(w_2w_1)^{2k}(0,0)\to (jl-1,N))|\Big)-\\
	-&\sum\limits_{k=0}^{\big[\frac{N-(j+2)l+1}{4l}\big]} Q_j(k)\Big(&|L_N((w_1w_2)^{2k+1}(0,0) \to (jl-1,N))|-\\
	& &-|L_N(w_1(w_2w_1)^{2k+1}(0,0)\to (jl-1,N))|\Big)\Big),
	\end{eqnarray*} 
	\begin{figure}[h!]
\label{FJ}
	\centerline{\includegraphics[width=450pt]{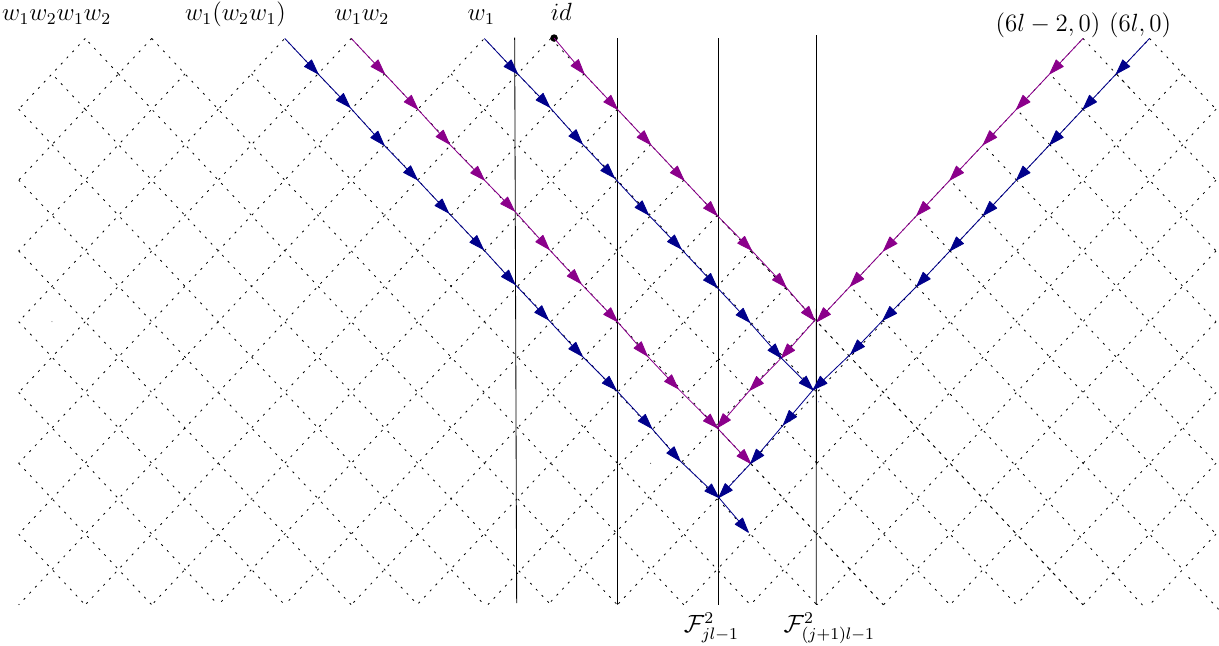}
}
	\caption{Counting weighted walks from $(0,0)$ to $(M,N)$ for $2l-1\leq M< 3l-1$ in the presence of filters $\mathcal{F}_{2l-1}^{1}$ and $\mathcal{F}_{3l-1}^{2}$ in terms of unweighted path counting.
The starting points are labeled by the corresponding reflection group elements acting on $(0,0)$.The violet paths contribute to $Z_+^0(3)$ and the blue paths to $Z_-^0(3)$.}
\end{figure}
To  obtain the multiplicity function for $j+1$-th strip we will take the same sets of path starting points $(w_1w_2)^i(0,0)$ and $w_1(w_1w_2)^i(0,0)$ as initial conditions  and place  another filter restriction $\mathcal{F}_{(j+1)l-1}^2$.
The weighted number of paths in the $j+1$-th strip will be given by
\begin{eqnarray}\label{zzz}
M^{j+1}_{(M,N)}(l)=2^{j-1}\Big(&\sum\limits_{k=0}^{\big[\frac{N-jl+1}{4l}\big]}P_j(k)\Big(Z_+^{2k}(j+1)-Z_-^{2k}(j+1)\Big)-\\
-&\sum\limits_{k=0}^{\big[\frac{N-(j+2)l+1}{4l}\big]}Q_j(k)\Big(Z_+^{2k+1}(j+1)-Z_-^{2k+1}(j+1)\Big)\Big),\nonumber
\end{eqnarray}
where

$Z_+^i(j+1)$ is the weighted number of paths from $(w_1w_2)^{i}(0,0)$ to $(M,N)$ for
$jl-1\leq M<(j+1)l-1$ in the presence of $\mathcal{F}_{jl-1}^1$ and $\mathcal{F}_{(j+1)l-1}^2$ , and 

$Z_-^i(j+1)$ is the weighted number  of paths  from $w_1(w_2w_1)^{i}(0,0)$ to $(M,N)$ in the presence of $\mathcal{F}_{jl-1}^1$  and  $\mathcal{F}_{(j+1)l-1}^2$. 

Now, we will write the expression for each summand explicitly in terms of binomial coefficients. It is clear that
$$Z_+^i(j+1)=
Z(L_N((-2il,0)\to (M,N)\;|\;\mathcal{F}_{jl-1}^1,\mathcal{F}_{(j+1)l-1}^{2})),$$  $$Z_-^i(j+1)=
Z(L_N((-2il-2,0)\to (M,N)\;|\;\mathcal{F}_{jl-1}^1,\mathcal{F}_{(j+1)l-1}^{2})).$$ 
Due to translation invariance similar to Remark  \ref{REM} we get
\begin{lemma}
The weighted number of lattice paths from $(-2il,0)$ to $(M,N)$ for $i\geq 0$ and with the filter restriction $\mathcal{F}_{l-1}^1$ and the filter restriction $\mathcal{F}_{2l-1}^2$ when $jl-1\leq M< (j+1)l-1$ is given by
\begin{equation}
\label{z1j}
Z_+^i(j+1)=\sum_{p=i}^{\bigl[\frac{N-jl+1}{2l}\bigl]} (-1)^{p-i}C^{(N)}_{M+2pl}+\sum_{p=i}^{\bigl[\frac{N-(j+1)l}{2l}\bigl]} (-1)^{p-i+1}C^{(N)}_{M-2(j+1+p)l+2},
\end{equation}
The weighted number of lattice paths  from $(-2il-2,0)$ to $(M,N)$  for $i\geq 0$  and with the filter restriction $\mathcal{F}_{l-1}^1$  and the filter restriction $\mathcal{F}_{2l-1}^2$ when $l-1\leq M< (j+1)l-1$ is given by
\begin{equation}
\label{z2j}
Z_-^i(j+1)=\sum_{p=i}^{\bigl[\frac{N-jl-1}{2l}\bigl]} (-1)^{p-i}C^{(N)}_{M+2pl-2}+\sum_{p=i}^{\bigl[\frac{N-(j+1)l-2}{2l}\bigl]} (-1)^{p-i+1}C^{(N)}_{M-2(j+1+p)l},
\end{equation}
where $C^{(N)}_M={N\choose\frac{N-M}{2}}$  is the number of unrestricted paths from $(0,0)$ to $(M,N)$.
\end{lemma}
Note that for $j+1=2$ (\ref{z1j}) and (\ref{z2j}) coincide with (\ref{th32}) and (\ref{th33}) respectively. The proof of this lemma is completely parallel to that of Theorem \ref{TH3}.

So we obtain
 \begin{equation*}
Z_+^i(j+1)-Z_-^i(j+1)=\sum_{p=i}^{\bigl[\frac{N-jl+1}{2l}\bigl]}(-1)^{p-i}F^{(N)}_{M+2pl}+\sum_{p=i}^{\bigl[\frac{N-(j+1)l}{2l}\bigl]} (-1)^{p-i}F^{(N)}_{M-2(j+1+p)l}.
\end{equation*}

Let us substitute the obtained expressions to (\ref{zzz}).

For the multiplicity function in the $j+1$-th strip we get 
\begin{eqnarray*}
\label{zzzk}
{M}^{j+1}_{(M,N)}(l)=2^{j-1}\sum_{k=0}^{\big[\frac{N-jl+1}{4l}\big]}P_j(k)\left(\sum_{p=2k}^{\bigl[\frac{N-jl+1}{2l}\bigl]}(-1)^{p-2k}F^{(N)}_{M+2pl}+\sum_{p=2k}^{\bigl[\frac{N-(j+1)l}{2l}\bigl]} (-1)^{p-2k}F^{(N)}_{M-2pl-2(j+1)l}\right)\\
-2^{j-1}\sum_{k=0}^{\big[\frac{N-(j+2)l+1}{4l}\big]}Q_j(k)\left(\sum_{p=2k+1}^{\bigl[\frac{N-jl+1}{2l}\bigl]}(-1)^{p-2k-1}F^{(N)}_{M+2pl}+\sum_{p=2k+1}^{\bigl[\frac{N-(j+1)l}{2l}\bigl]} (-1)^{p-2k-1}F^{(N)}_{M-2pl-2(j+1)l}\right).
\end{eqnarray*}

In the expression above we have the following set of terms
\begin{equation*}
F^{(N)}_{M},\; -F^{(N)}_{M+2l},\; F^{(N)}_{M+4l},\; \ldots,\; (-1)^{p}F^{(N)}_{M+2pl},\; \ldots,\; (-1)^{\bigl[\frac{N-jl+1}{2l}\bigl]}F^{(N)}_{M+2\bigl[\frac{N-jl+1}{2l}\bigl]l}
\end{equation*}
and, similarly,
\begin{equation*}
F^{(N)}_{M-2(j+1)l},\; \ldots,\; (-1)^{p}F^{(N)}_{M-2(j+1)l-2pl},\; \ldots,\; (-1)^{\bigl[\frac{N-(j+1)l}{2l}\bigl]}F^{(N)}_{M-2(j+1)l-2\bigl[\frac{N-(j+1)l}{2l}\bigl]l}
\end{equation*}
If we carefully recollect the terms, we get the following coefficients
\begin{equation*}
\Big(\sum_{n=0}^{\frac{p}{2}}P_j(n)+\sum_{n=0}^{\frac{p}{2}-1}Q_j(n)\Big)F^{(N)}_{M+2pl},\quad \text{for even $p$}
\end{equation*}
\begin{equation*}
-\Big(\sum_{n=0}^{\frac{p-1}{2}}P_j(n)+\sum_{n=0}^{\frac{p-1}{2}}Q_j(n)\Big)F^{(N)}_{M+2pl},\quad \text{for odd $p$}
\end{equation*}
\begin{equation*}
\Big(\sum_{n=0}^{\frac{p}{2}}P_j(n)+\sum_{n=0}^{\frac{p}{2}-1}Q_j(n)\Big)F^{(N)}_{M-2(j+1)l-2pl},\quad \text{for even $p$}
\end{equation*}
\begin{equation*}
-\Big(\sum_{n=0}^{\frac{p-1}{2}}P_j(n)+\sum_{n=0}^{\frac{p-1}{2}}Q_j(n)\Big)F^{(N)}_{M-2(j+1)l-2pl},\quad \text{for odd $p$}
\end{equation*}


	In order to obtain the desired result of the form
		\begin{eqnarray*}
	M^{j+1}_{(M,N)} = 2^{j-1}\Big(\sum_{k=0}^{\big[\frac{N-jl+1}{4l}\big]}P_{j+1}(k) F^{(N)}_{M+4kl}+\sum_{k=0}^{\big[\frac{N-(j+1)l}{4l}\big]}P_{j+1}(k)F^{(N)}_{M-4kl-2(j+1)l} -\\-\sum_{k=0}^{\big[\frac{N-(j+2)l+1}{4l}\big]}Q_{j+1}(k)F^{(N)}_{M+2l+4kl}-\sum_{k=0}^{\big[\frac{N-(j+3)l}{4l}\big]}Q_{j+1}(k)F^{(N)}_{M-4kl-2(j+2)l} \Big),
	\end{eqnarray*}
	it remains to check if the following recurrence relations
	\begin{equation}
	\label{p}
	P_{j+1}(k)=\sum_{n=0}^{k}P_{j}(n)+\sum_{n=0}^{k-1}Q_{j}(n),
	\end{equation}
	\begin{equation}
	\label{q}
	Q_{j+1}(k)=\sum_{n=0}^{k}P_{j}(n)+\sum_{n=0}^{k}Q_{j}(n).
	\end{equation}
    are satisfied by $P_j(k)$ and $Q_j(k)$, which we have assumed to hold for $j$-th strip:
		\begin{equation*}
	P_j(k)=\sum_{i=0}^{\big[\frac{j}{2}\big]}\binom{j-2}{2i}\binom{k-i+j-2}{j-2},\;\;\;\;\;
	Q_j(k)=\sum_{i=0}^{\big[\frac{j}{2}\big]}\binom{j-2}{2i+1}\binom{k-i+j-2}{j-2}.
	\end{equation*}

Indeed,\\
$$\sum_{n=0}^{k}P_{j}(n)+\sum_{n=0}^{k}Q_{j}(n)=\sum_{n=0}^{k}\sum_{i=0}^{\big[\frac{j}{2}\big]}\left(
\binom{j-2}{2i}+\binom{j-2}{2i+1}\right)\binom{n-i+j-2}{j-2}=$$
$$=\sum_{i=0}^{\big[\frac{j+1}{2}\big]}
\binom{j-1}{2i+1}\sum_{n=0}^{k}\binom{n-i+j-2}{j-2}
=\sum_{i=0}^{\big[\frac{j+1}{2}\big]}
\binom{j-1}{2i+1}\binom{k-i+j-1}{j-1}
= Q_{j+1}(k).
$$
The recurrence (\ref{p}) can be shown to hold in the same manner. 

Note that by the index $j$ this recurrence is similar to that of binomial coefficients, and by the variable $k$ it is similar to that of a number of integer points in tetrahedron \cite{LY}. 
 At the inductive step we have shown that the initial conditions of these recurrence relations are given by
	\begin{equation}
	\label{init}
	P_2(k) = 1,\quad Q_2(k) = 0\quad\forall k=0,1,\ldots
	\end{equation} 
The theorem is proven. 	
\end{proof}

\section{Conclusion}\label{conclusions}
In this paper we explored two-dimensional lattice path model with a periodic arrangement of multiple filter restrictions. We introduced two types of filters and counted numbers of paths descending from $(0,0)$ to $(M,N)$ in different configurations. We started by considering two filter restrictions and proceeded to multiple filters arranged periodically. Using reflection principle we obtained exact formulas for number of descending paths in considered configurations.\par
As it was mentioned earlier the filter appears naturally in the context of representation theory of quantum groups at roots of unity:
\begin{itemize}
	\item Lattice path model considered in the present paper and depicted in Figure \ref{figth3} can be used to obtain the model, weighted numbers of paths of which reproduce recurrence relations for the multiplicities in decomposition of tensor powers of fundamental representation of $U_q({\mathfrak{sl}_2})$ with divided powers, where $q$ is a root of unity. In order to do so one needs to consider all filters of type $1$ instead of the ones of type $2$ and add long steps. The resultant model is depicted in Figure \ref{LPMUQ}.
	\begin{figure}[h!]
	\centerline{\includegraphics[width=280pt]{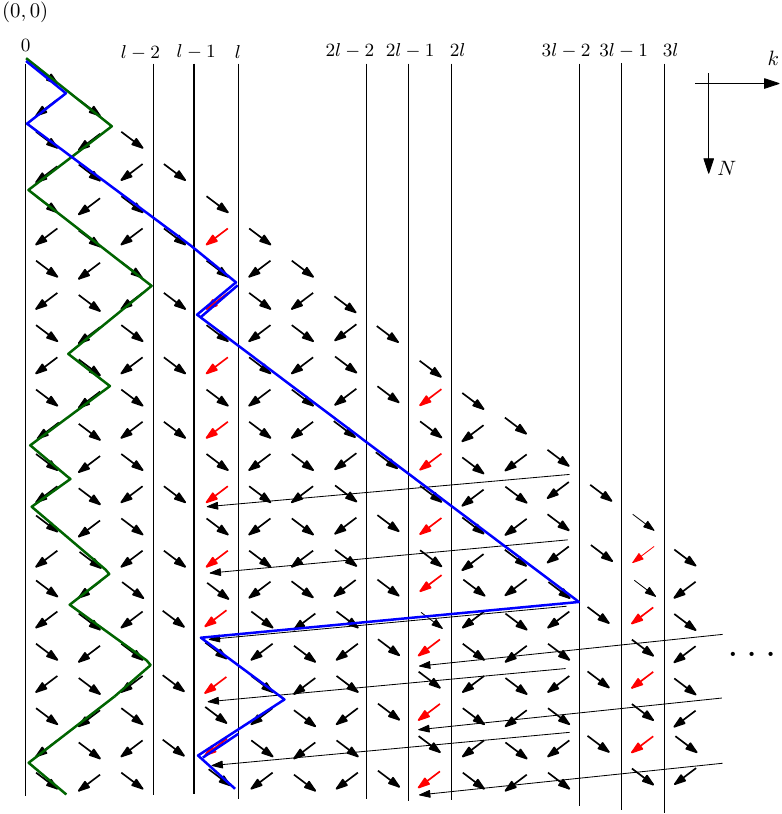}}
	\caption{Lattice path model reproducing multiplicities in tensor product decomposition of $U_q({\mathfrak{sl}_2})$ at roots of unity with divided powers. Horizontal axis corresponds to the highest weight of the component in tensor product decomposition, vertical axis corresponds to the considered tensor power. Examples of possible paths in such a walk highlighted in green and blue.}
	\label{LPMUQ}
	\end{figure}\par
	This can be seen as a folding transformation of the model considered in the present paper, which is schematically depicted in Figure \ref{folding1}.
	\begin{figure}[h!]
		\centerline{\includegraphics[width=250pt]{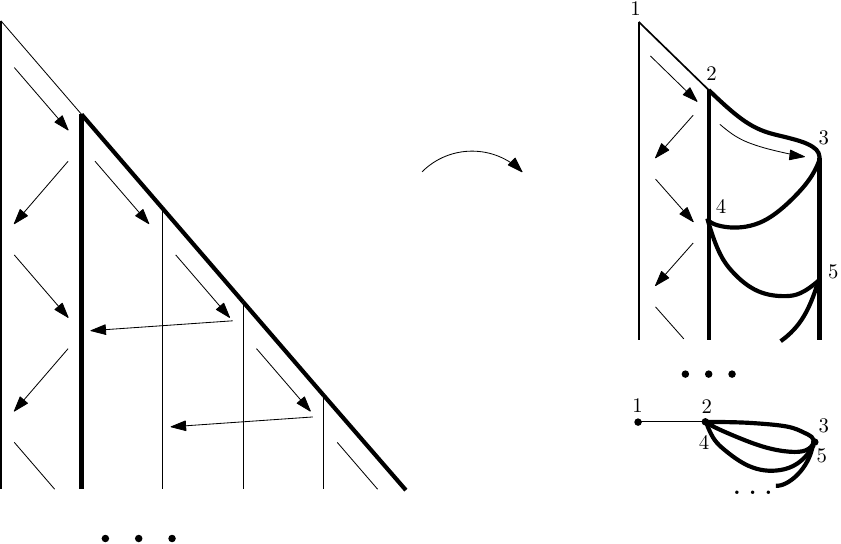}}
		\caption{Folding transformation of the lattice path model with periodic filters, leading to the model for $U_q({\mathfrak{sl}_2})$ with divided powers. Left caption is showing the initial lattice path model, right caption is showing the resultant one. Bottom-right caption is the view of the resultant lattice from above, where numbers are added for the purpose of showing layers.}
		\label{folding1}
	\end{figure}\par
	This model was studied in \cite{S} and its asymptotic analysis will be carried in \cite{LPRS}.
	\item Similarly to the application above, lattice path model considered in the present paper is also of use for deriving formulas for multiplicities in case of the small quantum group $u_q({\mathfrak{sl}_2})$. One needs to restrict it to the first two strips and add one sequence of long steps with double multiplicities. The resultant model is depicted in Figure \ref{LPMsmallUQ}.
		\begin{figure}[h!]
		\centerline{\includegraphics[width=250pt]{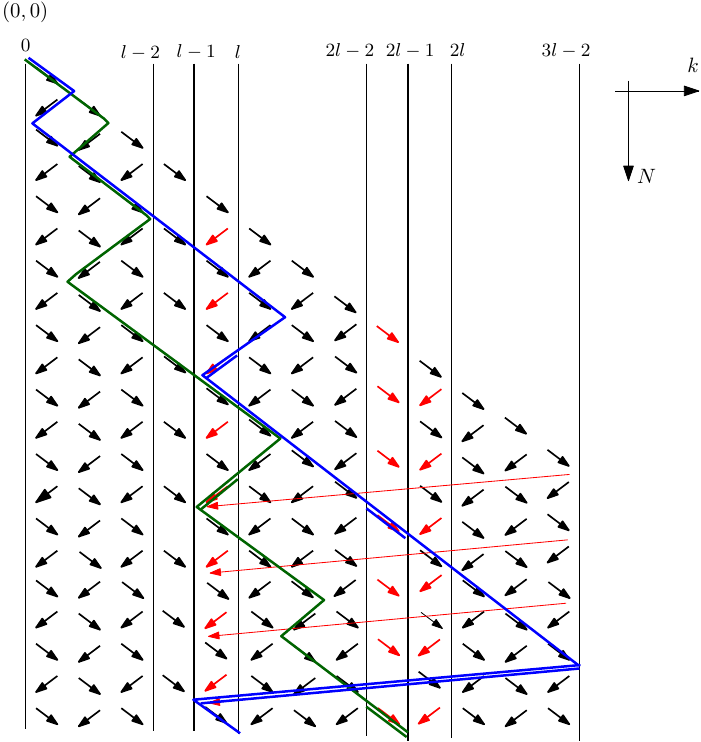}}
		\caption{Lattice path model reproducing multiplicities in tensor product decomposition of $u_q({\mathfrak{sl}_2})$. Horizontal axis corresponds to the highest weight of the component in tensor product decomposition, vertical axis corresponds to the considered tensor power. Examples of possible paths in such a walk highlighted in green and blue.}
		\label{LPMsmallUQ}
	\end{figure}\par
	This can be seen as a folding transformation of the model considered in the present paper with identification of the layers, which is schematically depicted in Figure \ref{folding2}.
	\begin{figure}[h!]
		\centerline{\includegraphics[width=250pt]{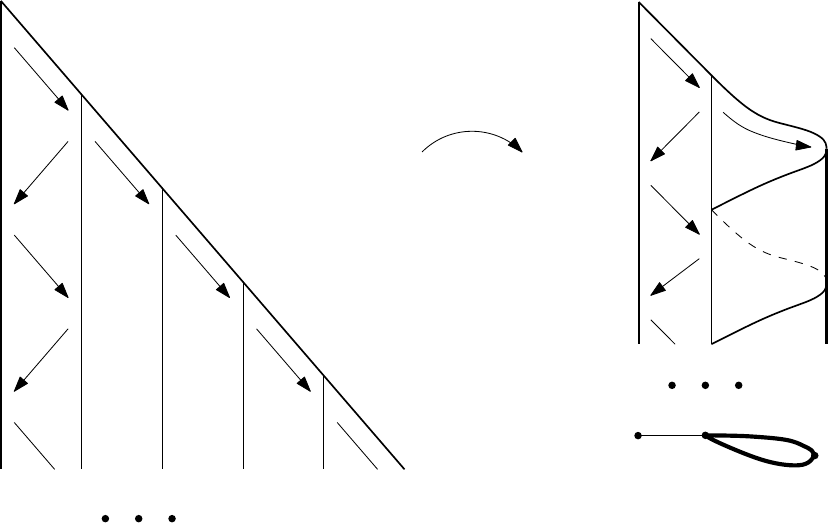}}
		\caption{Folding transformation of the lattice path model with periodic filters, leading to the model for $u_q({\mathfrak{sl}_2})$. Bottom-right caption is the view of the resultant lattice from above, all layers are identified into two.} 
		\label{folding2}
	\end{figure}\par
	Full analysis of this model will be carried in \cite{LPRS}.
	\item One can obtain $u_q({\mathfrak{sl}_2})$ from restricting $U_q({\mathfrak{sl}_2})$ with divided powers to $u_q^-U_q^0u_q^+$, where $u_q^\pm$ are subalgebras of the small quantum group $u_q({\mathfrak{sl}_2})$, generated by $F$ and $E$ resepctively, and $U^0_q$ is a central subalgebra of $U_q({\mathfrak{sl}_2})$, and then restricting $u_q^-U_q^0u_q^+$ to $u_q({\mathfrak{sl}_2})$. This gives another folding procedure for obtaining multiplicity formulas for the small quantum group $u_q({\mathfrak{sl}_2})$ from the lattice path model considered in the present paper.\par
	Firstly, one needs to branch the initial model into two identical ones, starting from the right boundary of the second strip, then proceed with adding long steps similarly to the model corresponding to $U_q({\mathfrak{sl}_2})$. As a result we obtain lattice path model reproducing multiplicities in tensor product decomposition of representations of $u_q^-U_q^0u_q^+$. This model is depicted in Figure \ref{LPMUUU}.
	\begin{figure}[h!]
		\centerline{\includegraphics[width=350pt]{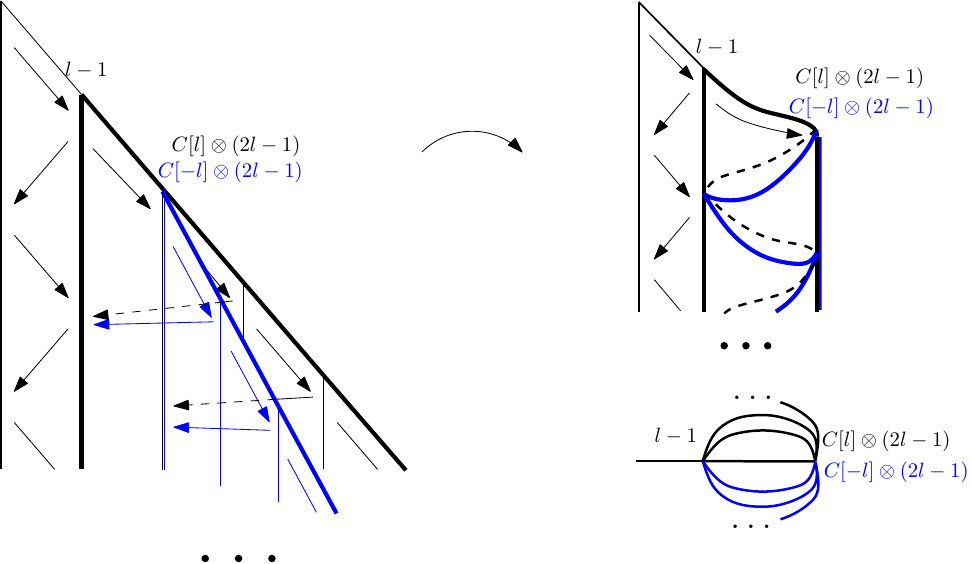}}
		\caption{Lattice path model reproducing multiplicities in tensor product decomposition of $u_q^-U_q^0u_q^+$. Left caption is showing branching, right caption shows folding. Bottom-right caption is the view of the resultant lattice from above.}
		\label{LPMUUU}
	\end{figure}\par
	Secondly, one needs to identify layers as in the rule depicted in Figure \ref{identif}, which corresponds to restriction of $u_q^-U_q^0u_q^+$ to $u_q({\mathfrak{sl}_2})$.
	\begin{figure}[h!]
		\centerline{\includegraphics[width=350pt]{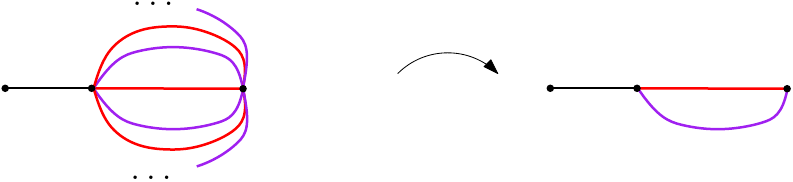}}
		\caption{Colors showing the rule of identification of the layers.}
		\label{identif}
	\end{figure}\par
 	This folding procedure gives the same formulas for multiplicities in tensor product decomposition for $u_q({\mathfrak{sl}_2})$, as the one depicted in Figure \ref{folding2}.
	\item Similar lattice path models emerge when studying category of tilting modules for $U_q({\mathfrak{sl}_2})$ when $q$ is an odd root of unity and the ground field is $\overline{\mathbb{F}_p}$ \cite{STWZ}.
	\item We expect that for lattice path models reproducing multiplicities in tensor product decomposition of $U_q(\mathfrak{sl}_n)$ at roots of unity, derivation of formulas for weighted numbers of paths will rely on similar combinatorial ideas: reflection principle involving Weyl group of the affine Kac-Moody algebra corresponding to $\mathfrak{sl}_n$. It is worth mentioning that obtaining such formulas explicitly is of interest for asymptotic representation theory, mainly, for constructing Plancherel measure and possibly obtaining its limit shape in different regimes, including regime when $n\to\infty$ (\cite{BOO},\cite{PR}, \cite{LPRS}).
\end{itemize}

\end{document}